\documentclass[12pt]{amsart}

\usepackage{amssymb}
\usepackage{amsthm}
\usepackage{amsmath}
\usepackage{a4wide}
\usepackage{latexsym}
\usepackage{mathrsfs}
\usepackage[v2,tips]{xy}
\usepackage{float}
\usepackage[T1]{fontenc}
\usepackage[latin1]{inputenc}
\usepackage{tikz}
\usetikzlibrary{plotmarks}

\newtheorem{theorem}{Theorem}[section]
\newtheorem{lemma}[theorem]{Lemma}
\newtheorem{proposition}[theorem]{Proposition}

\theoremstyle{definition}

\newtheorem{remark}[theorem]{Remark}

\theoremstyle{remark}

\numberwithin{equation}{section}

\newcounter{smallromans}

\newenvironment{romanenumerate}
{\begin{list}{{\normalfont\textrm{(\roman{smallromans})}}}%
    {\usecounter{smallromans}\setlength{\itemindent}{0cm}%
      \setlength{\leftmargin}{5.5ex}\setlength{\labelwidth}{5.5ex}%
      \setlength{\topsep}{0.75\parsep}\setlength{\partopsep}{0ex}%
      \setlength{\itemsep}{0ex}}}%
  {\end{list}}

\newcommand{\romanref}[1]{{\normalfont\textrm{(\ref{#1})}}}

\newcounter{smallalphs}

\newenvironment{alphenumerate}
{\begin{list}{{\normalfont\textrm{(\alph{smallalphs})}}}%
    {\usecounter{smallalphs}\setlength{\itemindent}{0cm}%
      \setlength{\leftmargin}{5.5ex}\setlength{\labelwidth}{5.5ex}%
      \setlength{\topsep}{0.75\parsep}\setlength{\partopsep}{0ex}%
      \setlength{\itemsep}{0ex}}}%
  {\end{list}}

\newcommand{\alphref}[1]{{\normalfont\textrm{(\ref{#1})}}}

\newcommand{\supp}{\operatorname{supp}}
\newcommand{\lin}{\operatorname{lin}}
\newcommand{\cllin}{\overline{\operatorname{lin}}\,}
\newcommand{\Sz}{\operatorname{Sz}}

\renewcommand{\leq}{\ensuremath{\leqslant}}
\renewcommand{\le}{\ensuremath{\leqslant}}

\renewcommand{\ge}{\ensuremath{\geqslant}}

\renewcommand{\phi}{\ensuremath{\varphi}}
\renewcommand{\epsilon}{\ensuremath{\varepsilon}}

\newcommand{\smashw}[2][l]{{\text{\makebox[0pt][#1]{$#2$}}}}

\begin{document}

\title[Uniqueness of the maximal ideal of
$\mathcal{B}(C([0,\omega_1\text{]}))$]{%
  \boldmath%
  Uniqueness of the maximal ideal of the Banach algebra of bounded
  operators on $C([0,\omega_1])$%
  \unboldmath}
\subjclass[2010]%
{Primary 47L10, 46H10; Secondary 47L20, 46B26, 47B38}

\author{Tomasz Kania and Niels Jakob Laustsen}
\address{Department of Mathematics and Statistics, Fylde College,
  Lancaster University, Lancaster LA1 4YF, United Kingdom}
\email{t.kania@lancaster.ac.uk, n.laustsen@lancaster.ac.uk}

\begin{abstract}
  Let $\omega_1$ be the first uncountable ordinal. By a result of
  Rudin, bounded operators on the Banach space~$C([0,\omega_1])$ have
  a natural re\-presen\-ta\-tion as \mbox{$[0,\omega_1]\times
    [0,\omega_1]$}-matrices. Loy and Willis observed that the set of
  opera\-tors whose final column is continuous when viewed as a
  scalar-valued function on $[0,\omega_1]$ defines a maximal ideal of
  co\-dimen\-sion one in the Banach alge\-bra
  $\mathscr{B}(C([0,\omega_1]))$ of bounded operators on
  $C([0,\omega_1])$.  We give a co\-ordinate-free characteri\-za\-tion
  of this ideal and deduce from it that $\mathscr{B}(C([0,\omega_1]))$
  contains no other maximal ideals. We then obtain a list of
  equivalent conditions describing the strictly smaller ideal of
  operators with separable range, and finally we investigate the
  structure of the lattice of all closed ideals of
  $\mathscr{B}(C([0,\omega_1]))$.
\end{abstract}
\maketitle

\section{Introduction}
\noindent
Loy and Willis~\cite{LW} proved that every derivation from the Banach
algebra~$\mathscr{B}(C([0,\omega_1]))$ of (bounded) opera\-tors on the
Banach space of continuous functions on the ordinal
inter\-val~$[0,\omega_1]$ equipped with its order topology into a
Banach $\mathscr{B}(C([0,\omega_1]))$-bimodule is auto\-matically
continuous. At the heart of their proof is the observation that the
set~$\mathscr{M}$ consisting of those operators whose final column is
continuous at~$\omega_1$ is a maximal ideal of codimension one in
$\mathscr{B}(C([0,\omega_1]))$. We call~$\mathscr{M}$ the
\emph{Loy--Willis ideal}. Its precise definition will be given in
Section~\ref{section2}, once we have introduced the necessary
terminology.

Motivated by the desire to understand the lattice of closed ideals of
$\mathscr{B}(C([0,\omega_1]))$, we shall prove the following result.
\begin{theorem}\label{uniquenessthm}
  The Loy--Willis ideal is the unique maximal ideal of
  $\mathscr{B}(C([0,\omega_1]))$.
\end{theorem} 

This result is in fact an immediate consequence of a more general
theorem, which is of in\-dependent interest because it gives a
coordinate-free characterization of the Loy--Willis ideal. (By
`coordinate-free', we mean without reference to the matrix
representation of operators.)

\begin{theorem}\label{coorfreecharofLWideal}
  An operator~$T$ on~$C([0,\omega_1])$ belongs to the Loy--Willis
  ideal~$\mathscr{M}$ if and only if the identity operator
  on~$C([0,\omega_1])$ does not factor through~$T$ in the sense that
  there are no operators~$R$ and~$S$ on~$C([0,\omega_1])$ such that $I
  = STR$.
\end{theorem}

The implication $\Rightarrow$ is obvious because the
ideal~$\mathscr{M}$ is proper.  The converse is much harder to prove;
this will be the topic of Section~\ref{sectionpfofmainthm}. Once it
has been proved, however, Theorem~\ref{uniquenessthm} is immediate
because Theorem~\ref{coorfreecharofLWideal} implies that the identity
operator belongs to the ideal generated by any operator not
in~$\mathscr{M}$.

Many Banach spaces~$X$ share with $C([0,\omega_1])$ the property that
the set
\[ \mathscr{M}_X = \{ T\in\mathscr{B}(X) : \text{the identity operator
  on}\ X\ \text{does not factor through}\ T\} \] is the unique maximal
ideal of~$\mathscr{B}(X)$. As noted in~\cite{dj}, the only non-trivial
part of this statement is that $\mathscr{M}_X$ is closed under
addition, and as in Theorem~\ref{coorfreecharofLWideal}, this is often
verified by showing that $\mathscr{M}_X$ is equal to some known ideal
of $\mathscr{B}(X)$. 

Banach spaces~$X$ for which $\mathscr{M}_X$ is the unique maximal
ideal of $\mathscr{B}(X)$ include:
\begin{romanenumerate}
\item $X = \ell_p$ for $1\le p <\infty$ and $X = c_0$
  (see~\cite{gmf});
\item $X = L_p([0,1])$ for $1\le p<\infty$ (see
  \cite[Theorem~1.3]{djs} and the text following it);
\item $X =\ell_\infty\cong L_\infty([0,1])$ (use
  \cite[Proposition~2.f.4]{lt}, as explained in~\cite[p.~253]{ll});
\item $X = \bigl(\bigoplus \ell_2^n\bigr)_{c_0}$ and $X =
  \bigl(\bigoplus \ell_2^n\bigr)_{\ell_1}$ (see \cite{llr} and
  \cite[Corollary~2.12]{lsz});
\item $X = C([0,1])$ (use \cite[Theorem~1]{pel} and
  \cite[Theorem~1]{ros}, as in \cite[Example~3.5]{brooker});
\item $X = C([0,\omega^\omega])$ and $X = C([0,\omega^\alpha])$, where
  $\alpha$ is a countable epsilon number, that is, a countable ordinal
  satisfying $\alpha = \omega^\alpha$. This result is due to Philip
  A.\ H.\ Brooker (unpublished), who has kindly given us permission to
  include it here together with the following proof. Let \mbox{$X =
    C([0,\omega^{\omega^\alpha}])$}, where $\alpha$ is either~$1$ or a
  countable epsilon number. The
  set~$\mathscr{S}\!\mathscr{Z}_\alpha(X)$ of operators on~$X$ having
  Szlenk index at most $\omega^\alpha$ is an ideal of $\mathscr{B}(X)$
  by \cite[Theorem~2.2]{brooker}. We shall discuss this ideal in more
  detail in Section~\ref{section5ideallattice}; for now, it suffices
  to note that $\mathscr{S}\!\mathscr{Z}_\alpha(X)\subseteq
  \mathscr{M}_X$ because the identity operator on~$X$ has Szlenk
  index~$\omega^{\alpha+1}$ (see Theorem~\ref{SzlenkCK}%
  \romanref{SzlenkCK1} below). Conversely, Bourgain
  \cite[Proposition~3]{bo} has shown that each
  operator~$T\notin\mathscr{S}\!\mathscr{Z}_\alpha(X)$ fixes a copy
  of~$X$. Hence, using \cite[Theorem~1]{pel} as above, we see that the
  identity operator on~$X$ factors through~$T$, so
  $\mathscr{M}_X\subseteq\mathscr{S}\!\mathscr{Z}_\alpha(X)$, and the
  conclusion follows.
\end{romanenumerate}
Note that, by~\cite{se}, $C([0,\omega_1])$ differs from all of the
above-mentioned Banach spaces by not being isomorphic to its Cartesian
square $C([0,\omega_1])\oplus C([0,\omega_1])$.  \medskip

Having thus understood the maximal ideal(s)
of~$\mathscr{B}(C([0,\omega_1]))$, we turn our atten\-tion to the
other closed ideals of this Banach algebra.  We begin with a
characterization of the ideal~$\mathscr{X}(C([0,\omega_1]))$ of
operators with separable range. To state it, we require three pieces
of notation.

Firstly, we associate with each countable ordinal~$\sigma$ the
multiplication opera\-tor $P_\sigma$ given by $P_\sigma f=
f\cdot\mathbf{1}_{[0,\sigma]}$ for $f\in C([0,\omega_1])$.  Since the
indicator function $\mathbf{1}_{[0,\sigma]}$ is idempotent and
continuous with norm one, $P_\sigma$ is a contractive projection
on~$C([0,\omega_1])$, and its range is isometrically isomorphic
to~$C([0,\sigma])$. For technical reasons (notably
Theorem~\ref{thmseprange}\alphref{thmseprange2e} below), we also
require the rank-one perturbation
\begin{equation}\label{defnPtilde}
  \widetilde{P}_\sigma = P_\sigma +
  \mathbf{1}_{[\sigma+1,\omega_1]}\otimes \epsilon_{\omega_1}
\end{equation}
of~$P_\sigma$, where $\epsilon_{\omega_1}\in C([0,\omega_1])^*$
denotes the point evaluation at~$\omega_1$.  Clearly
$\widetilde{P}_\sigma$ is a contractive projection.

Secondly, for Banach spaces $X$, $Y$ and $Z$, we let
\begin{equation}\label{defn_factorizationideal}  \mathscr{G}_Z(X,Y) =
  \lin\{ TS : S\in\mathscr{B}(X,Z),\,
  T\in\mathscr{B}(Z,Y)\}. \end{equation}  
This defines an operator ideal in the sense
of Pietsch, the \emph{ideal of operators factoring
  through}~$Z$. Note that if $Z$ contains a complemented copy of its
square $Z\oplus Z$, then the set $\{
TS : S\in\mathscr{B}(X,Z),\, T\in\mathscr{B}(Z,Y)\}$ is already closed
under addition, so the `$\lin$' in~\eqref{defn_factorizationideal} 
is super\-fluous. We write $\overline{\mathscr{G}}_Z(X,Y)$ for the norm
closure of $\mathscr{G}_Z(X,Y)$; this is a closed operator ideal.

Thirdly, we denote by~$c_0(\omega_1)$ the Banach space of
scalar-valued functions~$f$ defined on~$\omega_1=[0,\omega_1)$ such
that the set $\{\alpha\in[0,\omega_1) : |f(\alpha)|\ge\epsilon\}$ is
finite for each $\epsilon>0$, equipped with the pointwise-defined
vector-space operations and the supremum norm.

We can now state our characterization of the operators on
$C([0,\omega_1])$ with separable range. Its proof will be given in
Section~\ref{secion4}.

\begin{theorem}\label{thmseprange}
  The following five conditions are equivalent for an operator $T$ on
  $C([0,\omega_1])\colon$
  \begin{alphenumerate}
  \item\label{thmseprange2e} $T = \widetilde{P}_\sigma
    T\widetilde{P}_\sigma$ for some countable ordinal~$\sigma;$
  \item\label{thmseprange2g}
    $T\in\mathscr{G}_{C([0,\sigma])}(C([0,\omega_1]))$ for some
    countable ordinal~$\sigma;$
  \item\label{thmseprange2f}
    $T\in\overline{\mathscr{G}}_{C([0,\sigma])}(C([0,\omega_1]))$ for
    some countable ordinal~$\sigma;$
  \item\label{thmseprange2a} $T\in\mathscr{X}(C([0,\omega_1]));$
  \item\label{thmseprange2d} $T$ does not fix a copy
    of~$c_0(\omega_1)$.
  \end{alphenumerate}
\end{theorem}

\noindent 
\textsl{Warning!}  Theorem~\ref{thmseprange} does \textsl{not} imply
that the ideal $\mathscr{G}_{C([0,\sigma])}(C([0,\omega_1]))$ is
closed for each countable ordinal $\sigma$, despite the equivalence of
conditions~\alphref{thmseprange2g} and~\alphref{thmseprange2f}. The
reason is that, for given
$T\in\overline{\mathscr{G}}_{C([0,\tau])}(C([0,\omega_1]))$ (where
$\tau$ is a countable ordinal), the ordinal~$\sigma$ such
that~\alphref{thmseprange2g} holds may be much larger than~$\tau$ and
depend on~$T$.  \medskip

Finally, in Section~\ref{section5ideallattice}, we study the entire
lattice of closed ideals of $\mathscr{B}(C([0,\omega_1]))$. To
classify all the closed ideals of $\mathscr{B}(C([0,\omega_1]))$ seems
an impossible task. In the first instance, one would need to classify
the closed ideals of $\mathscr{B}(C([0,\omega^{\omega^\alpha}]))$ for
each countable ordinal~$\alpha$, something that already appears
intractable; it has currently been achieved only in the simplest case
$\alpha = 0$, where $C([0,\omega])\cong c_0$. 

Figure~\ref{diagramClosedIdeals} below summarizes the findings of
Section~\ref{section5ideallattice}, using the following conventions:
(i)~we suppress $C([0,\omega_1])$ everywhere, thus writing
$\mathscr{K}$ instead of $\mathscr{K}(C([0,\omega_1]))$ for the ideal
of compact operators on~$C([0,\omega_1])$, \emph{etc.};
(ii)~$\xymatrix{\mathscr{I}\,\ar@{^{(}->}[r]&\!\mathscr{J}}$ means
that the ideal $\mathscr{I}$ is properly contained in the ideal
$\mathscr{J}$; (iii)~a double-headed arrow indicates that there are no
closed ideals between~$\mathscr{I}$ and~$\mathscr{J}$; (iv)~$\alpha$
denotes a countable ordinal; and
(v)~$K_\alpha=[0,\omega^{\omega^\alpha}]$.
\begin{figure}[h]%
  $\spreaddiagramrows{-5.2ex}%
  \xymatrix{%
    &\mathscr{B}\\%
    \\ \\ \\ \\ %
    &\mathscr{M}\ar@{^{(}->>}[uuuuu]\\%
    \\ \\ \\ \\%
    &\mathscr{X}+\overline{\mathscr{G}}_{c_0\smashw{(\omega_1)}}%
    \ar@{^{(}->}[uuuuu]\\%
    \mathscr{X}\,\ar@{^{(}->}[ur]\\
    & \vdots\ar@{^{(}->}[uu]\\%
    \vdots\ar@{^{(}->}[uu]\\%
    &\overline{\mathscr{G}}_{C(K_{\alpha+1})\oplus c_0(\omega_1)}%
    \ar@{^{(}->}[uu]\\%
    \overline{\mathscr{G}}_{C(K_{\alpha+1})}%
    \ar@{^{(}->}[uu]\ar@{^{(}->}[ur]\\%
    &\overline{\mathscr{G}}_{C(K_{\alpha})\oplus c_0(\omega_1)}%
    \ar@{^{(}->}[uu]\\%
    \overline{\mathscr{G}}_{C(K_{\alpha})}%
    \ar@{^{(}->}[uu]\ar@{^{(}->}[ur]\\%
    & \vdots\ar@{^{(}->}[uu]\\%
    \vdots\ar@{^{(}->}[uu]\\%
    &\overline{\mathscr{G}}_{C(K_1)\oplus c_0(\omega_1)}%
    \ar@{^{(}->}[uu]\\%
    \overline{\mathscr{G}}_{C(K_1)}%
    \ar@{^{(}->}[uu]\ar@{^{(}->}[ur] &&
    \;\overline{\mathscr{G}}_{c_0(\omega_1)}%
    \ar@{_{(}->}[ul]\\%
    \\ \\ \\ \\ %
    &\;\overline{\mathscr{G}}_{c_0}%
    \ar@{_{(}->}[uuuuul]\ar@{^{(}->>}[uuuuur]\\%
    \\ \\ \\ \\ %
    &\mathscr{K}\ar@{^{(}->>}[uuuuu]\\%
    \\ \\ \\ \\%
    &\{0\}\ar@{^{(}->>}[uuuuu]}$
  \caption{Partial structure of the lattice of  closed ideals of
    $\mathscr{B} =
    \mathscr{B}(C([0,\omega_1]))$.}\label{diagramClosedIdeals} 
\end{figure}
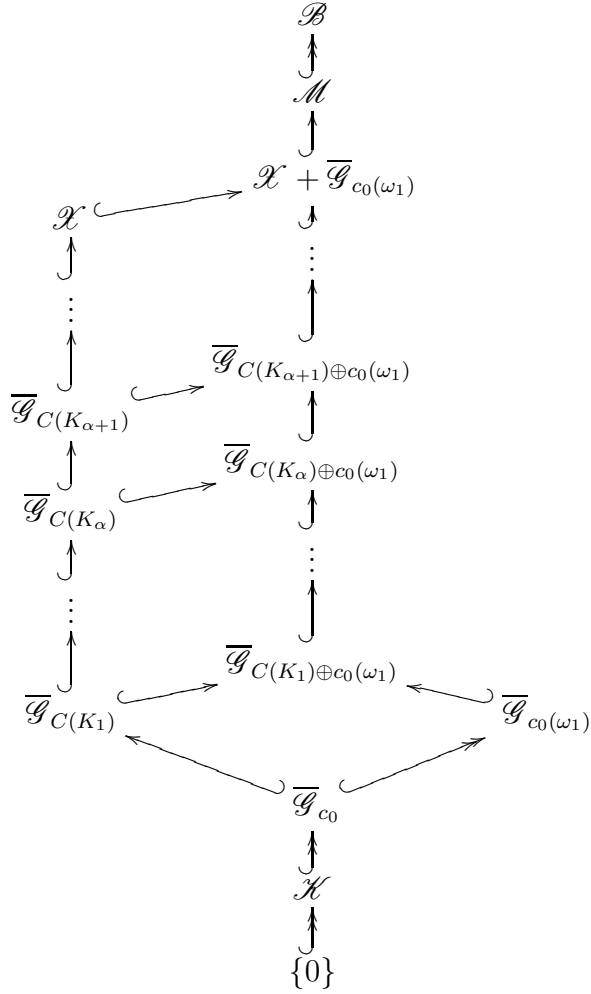

\section{Preliminaries}\label{section2}
\noindent
All Banach spaces are over the scalar field~$\mathbb{K}$, where
$\mathbb{K} = \mathbb{R}$ or $\mathbb{K} = \mathbb{C}$.  The term
\emph{ideal} always means two-sided ideal.  By an \emph{operator}, we
under\-stand a bounded linear operator between Banach spaces.  We
write~$\mathscr{B}(X)$ for the Banach algebra of all operators on the
Banach space~$X$, equipped with the operator norm.  Since
$\mathscr{B}(X)$ is unital, Krull's theorem implies that every proper
ideal of~$\mathscr{B}(X)$ is contained in a maximal ideal. It is well
known that every non-zero ideal of $\mathscr{B}(X)$ contains the ideal
$\mathscr{F}(X)$ of finite-rank operators on~$X$.

We define the \textit{support} of a scalar-valued function~$f$ defined
on a set~$K$ by $\supp(f)=\{k\in K : f(k)\neq 0\}$. When $K$ is a
compact space, $C(K)$ denotes the Banach space of all continuous
scalar-valued functions on $K$, equipped with the supremum norm. For
$k\in K$, the \emph{point evaluation} at~$k$ is the contractive
functional $\epsilon_k\in C(K)^*$ given by $\epsilon_k(f) = f(k)$.

The \emph{Kronecker delta} of a pair of ordinals $\alpha$ and~$\beta$
is given by $\delta_{\alpha,\beta} = 1$ if $\alpha = \beta$ and
$\delta_{\alpha,\beta} =0$ otherwise.  By convention, we consider~$0$
a limit ordinal.  For an ordi\-nal~$\sigma$, we write $[0,\sigma]$ for
the set of ordinals less than or equal to~$\sigma$, equipped with the
order topology.  This is a compact Hausdorff space which is metrizable
if and only if it is separable if and only if $\sigma$ is countable.
(As a set, $[0,\sigma]$ is of course equal to~$\sigma+1$; we use the
notation~$[0,\sigma]$ to emphasize that it is a topological space.)
The symbols~$\omega$ and~$\omega_1$ are reserved for the first
infinite and uncountable ordinal, respectively, while~$\mathbb{N}$
denotes the set of positive integers. We shall use extensively the
well-known fact that a scalar-valued function on $[0,\omega_1]$ is
continuous at~$\omega_1$ if and only if it is eventually constant.

Suppose that~$\sigma$ is an infinite ordinal, and let
$T\in\mathscr{B}(C([0,\sigma]))$. For each ordinal
\mbox{$\alpha\in[0,\sigma]$}, the functional $f\mapsto Tf(\alpha),\,
C([0,\sigma])\to\mathbb{K}$, is continuous, so by a result of
Rudin~\cite{Ru}, there are unique scalars $T_{\alpha,\beta}$, where
$\beta\in[0,\sigma]$, such that
\[ \sum_{\beta\in[0,\sigma]}|T_{\alpha,\beta}|<\infty\qquad
\text{and}\qquad Tf(\alpha) = \sum_{\beta\in [0,\sigma]}T_{\alpha,
  \beta}f(\beta)\qquad (f\in C([0,\sigma])). \] We can therefore
associate a $[0,\sigma]\times [0,\sigma]$-matrix $[T_{\alpha, \beta}]$
with~$T$. Note that the composition~$ST$ of operators~$S$ and $T$
on~$C([0,\sigma])$ corresponds to standard matrix multiplication in
the sense that
\begin{equation}\label{matrixmult}
  (ST)_{\alpha,\gamma}=\sum_{\beta\in [0,\sigma]}S_{\alpha,
    \beta}T_{\beta,\gamma}\qquad  (\alpha,\gamma\in [0,\sigma]).  
\end{equation}

We shall now specialize to the case where $\sigma = \omega_1$.  For
$T\in\mathscr{B}(C([0,\omega_1]))$ and $\alpha\in[0,\omega_1]$, we
denote by $r_\alpha^T$ and~$k_\alpha^T$ the $\alpha^{\text{th}}$ row
and column of the matrix of~$T$, respectively, considered as
scalar-valued functions defined on~$[0,\omega_1]$; thus
$r_\alpha^T\colon \beta\mapsto T_{\alpha,\beta}$ and $k_\alpha^T\colon
\beta\mapsto T_{\beta,\alpha}$. The following result of Loy and Willis
summarizes the basic properties of these functions.

\begin{proposition}\label{LoyWillis3.1}
  \emph{(\cite[Proposition~3.1]{LW})} Let $T$ be an operator
  on~$C([0,\omega_1])$. Then:
  \begin{romanenumerate}
  \item\label{LoyWillis3.1i} the function $r^T_\alpha$ is absolutely
    summable for each ordinal $\alpha\in[0,\omega_1]$, hence has
    countable support, and
    \[ \|T\|=\sup\biggl\{\sum_{\beta\in [0,\omega_1]}|T_{\alpha,
      \beta}| : \alpha\in [0,\omega_1]\biggr\}; \]
  \item\label{LoyWillis3.1ii} the function $k^T_\alpha$ is continuous
    whenever $\alpha=0$ or $\alpha$ is a countable successor ordinal;
  \item\label{LoyWillis3.1iii} the function $k^T_\alpha$ is continuous
    at $\omega_1$ for each countable ordinal~$\alpha;$
  \item\label{LoyWillis3.1iv} the restriction of $k^T_{\omega_1}$ to
    $[0,\omega_1)$ is continuous, and $\lim_{\alpha\to
      \omega_1}k^T_{\omega_1}(\alpha)$ exists.
  \end{romanenumerate}
\end{proposition}
Note that the statement in~\romanref{LoyWillis3.1iv} is the best
possible because the final column of the matrix associated with the
identity operator is equal to $\mathbf{1}_{\{\omega_1\}}$, so it is
not continuous at~$\omega_1$.

Loy and Willis studied the subspace $\mathscr{M}$
of~$\mathscr{B}(C([0,\omega_1]))$ consisting of those operators~$T$
such that $k^T_{\omega_1}$ is continuous at $\omega_1$. They observed
that $\mathscr{M}$ is an ideal of codimension one, hence maximal (see
\cite[p.~336]{LW}); this is the \emph{Loy--Willis ideal}.  It is
straightforward to verify that every operator on~$C([0,\omega_1])$ not
belonging to~$\mathscr{M}$ has uncountably many non-zero rows and
columns.
Although not required here, let us mention that the key result of Loy
and Willis \cite[Theorem~3.5]{LW} states that the ideal~$\mathscr{M}$
has a bounded right approximate identity.

\section{The Loy--Willis ideal: completion of the proof of
  Theorem~\ref{coorfreecharofLWideal}}\label{sectionpfofmainthm}
\noindent
In preparation for the proof of Theorem~\ref{coorfreecharofLWideal}
$(\Leftarrow)$, we require three lemmas.

\enlargethispage{10pt}
\begin{lemma}\label{fredholm}
  Let $T$ be a Fredholm operator acting on a Banach space~$X$ which is
  isomorphic to its hyperplanes (and hence to all its closed subspaces
  of finite co\-dimen\-sion). Then the identity operator on~$X$ factors
  through~$T$.
\end{lemma}
\begin{proof}
  Choose a closed subspace~$W$ of~$X$ which is complementary to~$\ker
  T$. Then $W$ has finite codimension in~$X$, so $W$ is isomorphic
  to~$X$ by assumption, and the restriction $\widetilde{T}\colon
  w\mapsto Tw,\, W\to T(X)$, is an iso\-mor\-phism, hence the identity
  operator on~$X$ factors through~$\widetilde{T}$. Now the result
  follows because $T(X)$ is complemented in~$X$, so 
  $\widetilde{T}$ factors through~$T$.
\end{proof}

\begin{lemma}\label{thePhiLemma} Let
  $\Xi=(\xi_\sigma)_{\sigma\in[0,\omega_1)}$ be a strictly increasing
  trans\-finite sequence of countable ordinals, and define
  $\xi_{\omega_1} = \omega_1$ and $\zeta_\lambda = \sup\{\xi_\sigma :
  \sigma\in[0,\lambda)\}$ for each limit ordinal
  \mbox{$\lambda\in[\omega,\omega_1]$}. Then:
  \begin{romanenumerate}
  \item\label{thePhiLemma1} the mapping $U_\Xi$ given by
    $U(\mathbf{1}_{[0,\sigma]}) = \mathbf{1}_{[0,\xi_\sigma]}$ for
    each $\sigma\in[0,\omega_1]$ extends unique\-ly to a linear isometry
    of $C([0,\omega_1])$ onto
    $\cllin\{\mathbf{1}_{[0,\xi_\sigma]} :
    \sigma\in[0,\omega_1]\};$\\
  \item\label{thePhiLemma2} \mbox{}\hspace{2.5cm}
    $\displaystyle{[0,\omega_1] =
      [0,\xi_0]\cup\bigcup_{\sigma\in[0,\omega_1)}[\xi_\sigma+1,
      \xi_{\sigma+1}]\cup\bigcup_{\lambda\in[\omega,\omega_1]\
        \text{\normalfont{limit}}} [\zeta_\lambda,\xi_\lambda],}$\\[1ex]
    where the intervals on the right-hand side are pairwise disjoint;
  \item\label{thePhiLemma3} the mapping $\phi_\Xi\colon
    [0,\omega_1]\to[0,\omega_1]$ given by
    \[ \mbox{}\quad \varphi_\Xi(\alpha) = \begin{cases} \xi_0 &
      \text{for}\ \alpha\in[0,\xi_0]\\
      \xi_{\sigma+1} & \text{for}\ \alpha\in [\xi_\sigma+1,
      \xi_{\sigma+1}],\ \text{where}\ \sigma\in[0,\omega_1),\\
      \zeta_\lambda & \text{for}\
      \alpha\in[\zeta_\lambda,\xi_\lambda],\ \text{where}\
      \lambda\in[\omega,\omega_1]\ \text{is a limit
        ordinal,} \end{cases} \] is continuous and satisfies
    $\phi_\Xi\circ\phi_\Xi =\phi_\Xi;$ hence the associated
    composition operator $\Phi_\Xi\colon f\mapsto f\circ \varphi_\Xi$
    defines a contractive projection of~$C([0,\omega_1])$ onto the
    subspace $\cllin\{\mathbf{1}_{[0,\xi_\sigma]} :
    \sigma\in[0,\omega_1]\};$
  \item\label{matrixPhiXi} the matrix associated with the operator
    $\Phi_\Xi$ is given by
    \[ \mbox{}\qquad(\Phi_\Xi)_{\alpha,\beta} = \begin{cases}
      \delta_{\beta,\xi_0} & \text{for}\
      \alpha\in[0,\xi_0]\\
      \delta_{\beta,\xi_{\sigma+1}} & \text{for}\
      \alpha\in[\xi_\sigma+1, \xi_{\sigma+1}],\ \text{where}\
      \sigma\in[0,\omega_1),\\
      \delta_{\beta,\zeta_\lambda} & \text{for}\
      \alpha\in[\zeta_\lambda,\xi_\lambda],\ \text{where}\
      \lambda\in[\omega,\omega_1]\ \text{is a limit
        ordinal.} \end{cases} \] 
  \end{romanenumerate}
\end{lemma}
 \begin{figure}[h]
    \begin{tikzpicture}[y=.2cm, x=.7cm,font=\sffamily, scale=0.7]
      {\color{white}\draw (0,0) -- coordinate (x axis mid) (10,0);}
      \draw (0,0) -- coordinate (y axis mid) (0,30); \draw (10,0) --
      coordinate (y axis mid) (10,30);
      \foreach \x in {7.9} \draw (\x,34) node[anchor=north]
      {$\zeta_\lambda$}; \foreach \x in {1} \draw (\x,34)
      node[anchor=north] {$\xi_0$}; \foreach \x in {2} \draw (\x,34)
      node[anchor=north] {$\xi_1$}; \foreach \x in {5.7} \draw (\x,34)
      node[anchor=north] {$\xi_2$};
\foreach \x in {6.7} \draw (\x,32.7)
      node[anchor=north] {$\ldots$};	
\foreach \x in {8.9} \draw (\x,32.7)
      node[anchor=north] {$\ldots$};	
      \foreach \x in {10} \draw (\x,33.4) node[anchor=north]
      {$\omega_1$}; \foreach \y in {0} \draw (0pt,\y) -- (0.3,\y)
      node[anchor=east] {$\omega_1\,\,$}; \draw (10,0) -- (9.7,0)
      node[anchor=west] {};
	
      \foreach \y in {30} \draw (0pt,\y) -- (0.3,\y) node[anchor=east]
      {$0\,\,$}; \draw (10,30) -- (9.7,30) node[anchor=west] {};
      \foreach \y in {25.5} \draw (-1.2,\y) node[anchor=west]
      {$\xi_0\,\,$}; \foreach \y in {21.5} \draw (-1.2,\y)
      node[anchor=west] {$\xi_1\,\,$};
	
      \foreach \y in {17} \draw (-1.2,\y) node[anchor=west]
      {$\xi_2\,\,$}; 
\foreach \y in {14.2} \draw (-1.0,\y) 
      node[anchor=west] {$\vdots$};
\foreach \y in {11} \draw (-1.0,\y) 
      node[anchor=west] {$\vdots$};
\foreach \y in {7} \draw (-1.2,\y) 
      node[anchor=west] {$\xi_\lambda\,\,$};
\foreach \y in {4} \draw (-1.0,\y) 
      node[anchor=west] {$\vdots$};
      \draw plot[mark=*, mark options={fill=black}] (1, 30) -- (1,
      26); \draw plot[mark=*] (1, 26); \draw plot[mark=*] (9.9, 0.35);
      \draw plot[mark=*, mark options={fill=white}] (2, 26) -- (2,
      22); \draw plot[mark=*] (2, 22); \draw plot[mark=*, mark
      options={fill=white}] (5.7, 22) -- (5.7, 17); \draw plot[mark=*]
      (5.7, 17); \draw plot[mark=*, mark options={fill=black}] (7.9,
     11) -- (7.9, 7); \draw plot[mark=*, mark options={fill=black}]
      (7.9, 7);
    \end{tikzpicture} \caption{Structure of the the matrix associated
      with $\Phi_\Xi$.}\label{figRxmatrix}
  \end{figure}
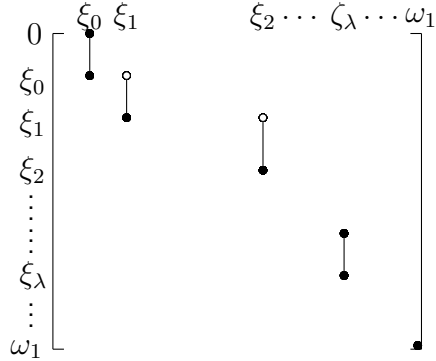
\begin{proof}
  \romanref{thePhiLemma1}.  For $n\in\mathbb{N}$, scalars $c_1,\ldots,
  c_n$ and ordinals $0\le \sigma_1< \cdots< \sigma_n\leq \omega_1$, we
  have
  \[\biggl\|\sum_{j=1}^n c_j \mathbf{1}_{[0,\sigma_j]}\biggr\|
  =\max_{1\le m\le n}\biggl|\sum_{j=m}^n c_j\biggr|=
  \biggl\|\sum_{j=1}^n c_j \mathbf{1}_{[0,\xi_{\sigma_j}]}\biggr\|. \]
  Hence $U_\Xi$ defines a linear 
isometry of
  $\lin\{\mathbf{1}_{[0,\sigma]} : \sigma\in[0,\omega_1]\}$ on\-to
  \mbox{$\lin\{\mathbf{1}_{[0,\xi_\sigma]} :
    \sigma\in[0,\omega_1]\}$}.  Now the conclusion follows because the
  domain of $U_\Xi$ is dense in $C([0,\omega_1])$.

  \romanref{thePhiLemma2}. This is straightforward to verify.

  \romanref{thePhiLemma3}. Clause~\romanref{thePhiLemma2} ensures that
  the definition of $\phi_\Xi$ makes sense.  To prove that~$\phi_\Xi$
  is continuous, suppose that $(\alpha_j)$ is a net in $[0,\omega_1]$
  which converges to~$\alpha$. By the definition of the order
  topology, this means that for each $\beta<\alpha$, we have
  $\beta<\alpha_j\le\alpha$ eventually.  If $\alpha\notin
  \{\zeta_\lambda : \lambda\in[\omega,\omega_1]\ \text{limit}\}$, then
  $\phi_\Xi(\alpha_j) = \phi_\Xi(\alpha)$ eventually, so the
  continuity of $\phi_\Xi$ at~$\alpha$ is clear in this
  case. Otherwise $\alpha = \zeta_\lambda$ for some limit ordinal
  $\lambda\in[\omega,\omega_1]$, and $\phi_\Xi(\alpha) =
  \zeta_\lambda$. Given $\beta<\zeta_\lambda$, we can take
  $\sigma\in[0,\lambda)$ such that $\beta<\xi_\sigma$. Since
  $\xi_\sigma<\xi_{\sigma+1}\le\zeta_\lambda$, we have
  $\xi_\sigma<\alpha_j\le\zeta_\lambda$ eventually. Hence the
  definition of~$\phi_\Xi$ implies that
  $\xi_\sigma<\phi_\Xi(\alpha_j)\le\zeta_\lambda$ eventually, so that
  $\lim_j\phi_\Xi(\alpha_j) = \zeta_\lambda = \phi_\Xi(\alpha$), as
  required.

  We have $\phi_\Xi\circ\phi_\Xi =\phi_\Xi$ because by definition
  $\phi_\Xi(\alpha)$ belongs to the same interval as $\alpha$ in the
  partition of~$[0,\omega_1]$ given in~\romanref{thePhiLemma2}.  Hence
  $\Phi_\Xi$ is a contractive projection. 

  To determine its range, we observe that
  \begin{equation}\label{eqPhiXi}
    \Phi_\Xi(\mathbf{1}_{[0,\alpha]}) = \begin{cases} 0 &\text{for}\
      \alpha\in[0,\xi_0)\\ \mathbf{1}_{[0,\xi_\sigma]} &\text{for}\
      \alpha\in[\xi_0,\omega_1],\ \text{where}\ \sigma =
      \sup\{\tau\in[0,\omega_1]:\xi_\tau\le\alpha\}. \end{cases} 
  \end{equation}  
  Consequently, we have
  \begin{equation}\label{eqRangePhiXi}
    \Phi_\Xi(C([0,\omega_1]))\subseteq
    \cllin\{\mathbf{1}_{[0,\xi_\sigma]} : \sigma\in[0,\omega_1]\}
  \end{equation}
  because the linear span of $\{\mathbf{1}_{[0,\alpha]} :
  \alpha\in[0,\omega_1]\}$ is dense in~$C([0,\omega_1])$.

  Conversely, \eqref{eqPhiXi} implies that
  $\Phi_\Xi(\mathbf{1}_{[0,\xi_\sigma]}) =
  \mathbf{1}_{[0,\xi_\sigma]}$ for each $\sigma\in[0,\omega_1]$, so we
  have equality in~\eqref{eqRangePhiXi} because $\Phi_\Xi$ has closed
  range.

  \romanref{matrixPhiXi}. This is clear from the definition
  of~$\phi_\Xi$.
\end{proof}

\begin{lemma}\label{anothercompop}
  Let $H$ be an uncountable subset of~$[0,\omega_1]$.  Then $H$ is
  order-iso\-mor\-phic to~$[0,\omega_1]$, and the order isomorphism
  $\psi_H\colon [0,\omega_1]\to H$ is continuous with respect to the
  relative topology on~$H$ if and only if~$H$ is closed
  in~$[0,\omega_1]$.

  Now suppose that~$H$ is closed in~$[0,\omega_1]$.  Then $\omega_1\in
  H$, and the composition opera\-tor $\Psi_H\colon f\mapsto f\circ
  \iota_H\circ\psi_H$, where $\iota_H\colon H\to [0,\omega_1]$ denotes
  the inclusion mapping, defines a contractive operator
  on~$C([0,\omega_1])$.
\end{lemma}
\begin{proof} Clearly $H$ is order-iso\-mor\-phic to~$[0,\omega_1]$.
  If the order isomorphism $\psi_H$ is continuous, then $H$ is compact
  (as the continuous image of the compact 
  space~$[0,\omega_1]$) and hence closed in~$[0,\omega_1]$.

  Conversely, suppose that~$H$ is closed in~$[0,\omega_1]$. Then
  $\psi_H$ is a bijection between two compact Hausdorff spaces, so
  $\psi_H$ is continuous if and only if its inverse is. Now
  \[ \psi_H([0,\sigma)) = [0,\psi_H(\sigma))\cap H\qquad
  \text{and}\qquad \psi_H((\sigma,\omega_1]) =
  (\psi_H(\sigma),\omega_1]\cap H\qquad (\sigma\in[0,\omega_1]), \]
  which shows that $\psi_H^{-1}$ is continuous because the sets
  $[0,\sigma)$ and $(\sigma,\omega_1]$ for $\sigma\in[0,\omega_1]$
  form a subbasis for the topology of~$[0,\omega_1]$.

  The second part of the lemma follows immediately.
\end{proof}

Un\-like~$\Phi_\Xi$, the matrix associated with $\Psi_H$ cannot in
general be depicted schematically; it is, however, possible in the
particular case that we shall consider in the proof of
Theorem~\ref{coorfreecharofLWideal}, as shown in Figure~\ref{fig2}
below.

\begin{proof}[Proof of Theorem~\ref{coorfreecharofLWideal}
  $(\Leftarrow)$]
  Let $T\in \mathscr{B}(C([0,\omega_1]))\setminus \mathscr{M}$.  Going
  through a series of reductions, we shall eventually reach the
  conclusion that there are operators
  \mbox{$R,S\in\mathscr{B}(C([0,\omega_1]))$} and
  $F\in\mathscr{F}(C([0,\omega_1]))$ such that $STR + F = I$. Then
  $STR = I - F$ is a Fredholm operator, and the conclusion follows
  from Lemma~\ref{fredholm}.

  We begin by reducing to the case where each column with countable
  index of the associated matrix vanishes eventually. Indeed, since
  $r^T_{\omega_1}$ is absolutely summable, we can take a countable
  ordinal~$\rho$ such that $T_{\omega_1,\beta} = 0$ whenever $\beta\in
  (\rho,\omega_1)$.
  Proposition~\ref{LoyWillis3.1}\romanref{LoyWillis3.1iii} then
  implies that $k^T_\beta$ is eventually null for each $\beta\in
  (\rho, \omega_1)$, and hence the $\beta^{\text{th}}$ column of the
  operator $T_1 = T(I-P_\rho)$ is eventually null for each
  $\beta\in[0,\omega_1)$. Note, moreover, that $T_1\notin\mathscr{M}$
  because $k_{\omega_1}^{T_1} = k_{\omega_1}^T$.

  Next, perturbing $T_1$ by a finite-rank operator and rescaling, we
  can arrange that the final row and column of its matrix are equal
  to~$\mathbf{1}_{\{\omega_1\}}$. To verify this, we observe that
  Propo\-si\-tion~\ref{LoyWillis3.1}\romanref{LoyWillis3.1iv} implies
  that the function $g\colon [0,\omega_1]\to\mathbb{K}$ given by
  \[ g(\alpha) = \begin{cases} (T_1)_{\alpha, \omega_1} & \text{for}\
    \alpha\in[0,\omega_1)\\ \lim_{\gamma\to\omega_1}(T_1)_{\gamma,
      \omega_1} &\text{for}\ \alpha=\omega_1 \end{cases} \] is
  continuous, so $G=g\otimes \epsilon_{\omega_1}$ defines a
  finite-rank operator.  The number $c = (T_1)_{\omega_1, \omega_1}-
  g(\omega_1)$ is non-zero because $T_1\notin\mathscr{M}$, and the
  operator $T_2 = c^{-1}(T_1-G)$ satisfies $k^{T_2}_\beta = c^{-1}
  k^{T_1}_\beta$ for each $\beta\in[0,\omega_1)$ and
  $k^{T_2}_{\omega_1} = \mathbf{1}_{\{\omega_1\}}$.  The latter
  statement implies that $T_2\notin\mathscr{M}$, and
  $r^{T_2}_{\omega_1} = \mathbf{1}_{\{\omega_1\}}$ because
  $k^{T_1}_\beta$ vanishes eventually for each $\beta\in[0,\omega_1)$.

  We shall now inductively construct two transfinite sequences
  $(\eta_\sigma)_{\sigma\in[0,\omega_1)}$ and
  $(\xi_\sigma)_{\sigma\in[0,\omega_1)}$ of countable ordinals such
  that $\eta_\tau+\omega < \eta_\sigma$ and $\xi_\tau<\xi_\sigma$
  whenever $\tau<\sigma$. First, let $\eta_0=\xi_0=0$.  Next, assuming
  that the sequences $(\eta_\tau)_{\tau\in[0,\sigma)}$ and
  $(\xi_\tau)_{\tau\in[0,\sigma)}$ have been chosen for some
  $\sigma\in[1,\omega_1)$, we define
  \begin{equation}\label{defnetasigma} \eta_\sigma = \begin{cases}
      \displaystyle{\sup\biggl(\{\eta_\tau +
        \omega\}\cup\bigcup_{\beta\in[0,\xi_\tau]}\supp
        (k_\beta^{T_2})\biggr) + 1}
      &\text{for}\ \sigma = \tau+1,\ \text{where}\ \tau\in[0,\omega_1),\\
      \sup\{\eta_\tau: \tau\in[0,\sigma)\} &\text{for}\ \sigma\
      \text{a limit ordinal} \end{cases} \end{equation} and
  \begin{equation}\label{defnxisigma}  \xi_\sigma =
    \sup\biggl(\{\xi_\tau + 1 :
    \tau\in[0,\sigma)\}\cup\bigcup_{\alpha\in[0,\eta_\sigma+\omega]} 
    \supp(r_\alpha^{T_2})\biggr). \end{equation} It is clear that
  $\xi_\tau<\xi_\sigma$ for each $\tau<\sigma$, and also that
  $\eta_\tau+\omega < \eta_\sigma$ if $\sigma$ is a successor
  ordinal. On the other hand, if $\sigma$ is a limit ordinal, then
  $\tau<\sigma$ implies that $\tau+1<\sigma$, so $\eta_\tau+\omega <
  \eta_{\tau+1}\le \eta_\sigma$, as desired. Hence the induction
  continues.

  Let $T_3 = T_2\Phi_\Xi$, where $\Phi_\Xi$ is the composition
  operator associated with the trans\-finite sequence $\Xi =
  (\xi_\sigma)_{\sigma\in[0,\omega_1)}$ as in
  Lemma~\ref{thePhiLemma}\romanref{thePhiLemma3}. Using
  Lemma~\ref{thePhiLemma}\romanref{matrixPhiXi} and matrix
  multi\-pli\-ca\-tion, we see that
  $r^{T_3}_{\omega_1}=k^{T_3}_{\omega_1} =
  \mathbf{1}_{\{\omega_1\}}$. In fact, each of the rows of the matrix
  of~$T_3$ indexed by the set
  $H=\bigcup_{\sigma\in[1,\omega_1)}[\eta_\sigma,
  \eta_{\sigma}+\omega]\cup\{\omega_1\}$ has (at most) one-point
  support.  More precisely, since the sets defining~$H$ are pairwise
  disjoint, we can define a map $\theta\colon H\to[1,\omega_1]$ by
  \[ \theta(\alpha) = \begin{cases} \xi_\sigma &\text{for}\ \alpha\in
    [\eta_\sigma, \eta_{\sigma}+\omega],\ \text{where}\
    \sigma\in[1,\omega_1)\ \text{is  a successor ordinal,}\\
    \zeta_\sigma &\text{for}\ \alpha\in [\eta_\sigma,
    \eta_{\sigma}+\omega],\ \text{where}\ \sigma\in[1,\omega_1)\
    \text{is a limit ordinal,}\\
    \omega_1 &\text{for}\ \alpha=\omega_1, \end{cases} \] where
  $\zeta_\sigma = \sup\{\xi_\tau : \tau\in[0,\sigma)\}$ as in
  Lemma~\ref{thePhiLemma}, and we claim that
  \begin{equation}\label{supprT3} \supp(r_\alpha^{T_3})\subseteq
    \{\theta(\alpha)\}\qquad (\alpha\in H). \end{equation}
  This has already been verified for $\alpha = \omega_1$. 
  Otherwise $\alpha\in [\eta_\sigma, \eta_{\sigma}+\omega]$ for some
  $\sigma\in[1,\omega_1)$, and   $\omega_1\notin\supp(r_\alpha^{T_3})$
  because $k^{T_3}_{\omega_1} =
  \mathbf{1}_{\{\omega_1\}}$. Given $\gamma\in[0,\omega_1)$, matrix
  multiplication shows that
  \begin{equation*}
    (T_3)_{\alpha,\gamma} =
    \sum_{\beta\in[0,\omega_1]} 
    (T_2)_{\alpha,\beta} (\Phi_\Xi)_{\beta,\gamma} =
    \sum_{\beta\in[0,\xi_\sigma]} (T_2)_{\alpha,\beta}
    (\Phi_\Xi)_{\beta,\gamma} \end{equation*} because
  $\alpha\le\eta_{\sigma}+\omega$ implies that
  $\sup\supp(r_\alpha^{T_2})\le\xi_\sigma$
  by~\eqref{defnxisigma}, so that
  $(T_2)_{\alpha,\beta} = 0$  for \mbox{$\beta\in(\xi_\sigma,\omega_1]$}.  
  Now if $\sigma$ is a successor ordinal, say $\sigma = \tau+1$, then
  for each $\beta\in[0,\xi_\tau]$, we have
  $\sup\supp(k_\beta^{T_2})<\eta_\sigma\le\alpha$
  by~\eqref{defnetasigma}, so that $(T_2)_{\alpha,\beta} = 0$ for
  such~$\beta$, and hence
  \[ (T_3)_{\alpha,\gamma} = \sum_{\beta\in[\xi_\tau+1,\xi_{\tau+1}]}
  (T_2)_{\alpha,\beta}(\Phi_\Xi)_{\beta,\gamma} = \begin{cases}
    \displaystyle{\sum_{\beta\in[\xi_\tau+1,\xi_{\tau+1}]}
      (T_2)_{\alpha,\beta}}
    &\text{if}\ \gamma = \xi_{\tau+1} = \xi_\sigma = \theta(\alpha)\\
    \qquad0 &\text{otherwise} \end{cases} \] by
  Lemma~\ref{thePhiLemma}\romanref{matrixPhiXi}.  Otherwise $\sigma$
  is a limit ordinal, and for each $\beta\in[0,\zeta_\sigma)$, we can
  choose $\tau\in[0,\sigma)$ such that $\beta\le\xi_\tau$. Then
  $\sup\supp(k_\beta^{T_2})<\eta_{\tau+1}<\eta_\sigma\le\alpha$, so
  that $(T_2)_{\alpha,\beta} = 0$ for such~$\beta$, and as above we
  find that
  \[ (T_3)_{\alpha,\gamma} = \sum_{\beta\in[\zeta_\sigma,\xi_\sigma]}
  (T_2)_{\alpha,\beta}(\Phi_\Xi)_{\beta,\gamma} = \begin{cases}
    \displaystyle{\sum_{\beta\in[\zeta_\sigma,\xi_\sigma]}
      (T_2)_{\alpha,\beta}}
    &\text{if}\ \gamma = \zeta_\sigma = \theta(\alpha)\\
    \qquad0 &\text{otherwise.} \end{cases} \] This completes the proof
  of~\eqref{supprT3}.

  The set $H$ defined above is clearly uncountable. To prove that it
  is also closed, let $(\alpha_j)$ be a net in~$H$ converging to some
  $\alpha\in[0,\omega_1]$. Then, for each $\beta\in[0,\alpha)$, there
  is~$j_0$ such that $\beta<\alpha_j\le\alpha$ whenever $j\ge j_0$. In
  particular, we may suppose that $\alpha_j\le\alpha$ for
  each~$j$. Let $\sigma = \sup\{\tau\in[1,\omega_1) :
  \eta_\tau\le\alpha\}\in [1,\omega_1]$. If $\sigma = \omega_1$, then
  $\alpha\ge\sup\{\eta_\tau : \tau\in[0,\omega_1)\} = \omega_1$, so
  that $\alpha = \omega_1\in H$. Otherwise $\sigma$ is countable. The
  choice of~$\sigma$ implies that
  $\eta_\sigma\le\alpha<\eta_{\sigma+1}$.  (In the case where $\sigma$
  is a limit ordinal, the first inequality follows from the fact that
  $\eta_\sigma = \sup\{\eta_\tau : \tau\in[0,\sigma)\}$
  by~\eqref{defnetasigma}.)  Hence, for each~$j$, we have
  \[ \alpha_j\in H\cap[0,\alpha]\subseteq H\cap[0,\eta_{\sigma+1}) =
  \bigcup_{\tau\in[1,\sigma]} [\eta_\tau,\eta_\tau+\omega]\subseteq
  [0,\eta_\sigma+\omega], \] so $\eta_\sigma+\omega\ge \lim_j\alpha_j
  = \alpha$ and thus $\alpha\in [\eta_\sigma,\eta_\sigma
  +\omega]\subseteq H$, as desired.

  We can therefore associate with~$H$ the composition
  opera\-tor~$\Psi_H$ as in Lemma~\ref{anothercompop};
  Figure~\ref{fig2} sketches the matrix associated with $\Psi_H$.
  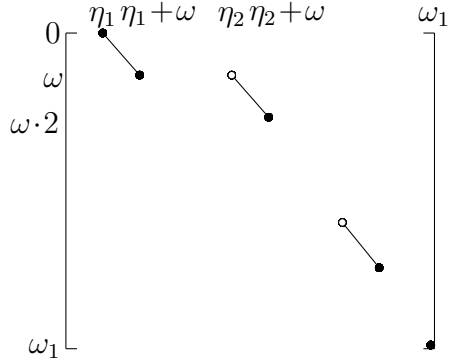
\begin{figure}[h]
    \begin{tikzpicture}[y=.2cm, x=.7cm,font=\sffamily,scale=0.7]
      {\color{white}\draw (0,0) -- coordinate (x axis mid) (10,0);}
      \draw (0,0) -- coordinate (y axis mid) (0,30); \draw (10,0) --
      coordinate (y axis mid) (10,30);
      \foreach \x in {8.5} \draw (\x,34.7) node[anchor=north] {};
      \foreach \x in {1} \draw (\x,33.5) node[anchor=north] {$\eta_1$};
      \foreach \x in {2.5} \draw (\x,34) node[anchor=north]
      {$\eta_1\!+\!\omega$}; \foreach \x in {4.5} \draw (\x,33.5)
      node[anchor=north] {$\eta_2$}; \foreach \x in {6} \draw (\x,34)
      node[anchor=north] {$\eta_2\!+\!\omega$}; \foreach \x in {10} \draw
      (\x,33.5) node[anchor=north] {$\omega_1$}; \foreach \y in {0}
      \draw (0pt,\y) -- (0.3,\y) node[anchor=east] {$\omega_1\,$};
      \draw (10,0) -- (9.7,0) node[anchor=west] {};
	
      \foreach \y in {30} \draw (0pt,\y) -- (0.3,\y) node[anchor=east]
      {$0\,$}; \draw (10,30) -- (9.7,30) node[anchor=west] {};
      \foreach \y in {25.5} \draw (-0.9,\y) node[anchor=west]
      {$\omega\,\,$}; \foreach \y in {21.5} \draw (-1.8,\y) 
      node[anchor=west] {$\omega\!\cdot\! 2\,\,$}; \foreach \y in {7.7}
      \draw (-1.2,\y) node[anchor=west] { };
      \draw plot[mark=*, mark options={fill=black}] (1, 30) -- (2,
      26); \draw plot[mark=*, mark options={fill=black}] (9.9, 0.35);
      \draw plot[mark=*, mark options={fill=white}] (4.5, 26) -- (5.5,
      22); \draw plot[mark=*] (5.5, 22); \draw plot[mark=, mark
      options={fill=black}] (7.5,12) -- (8.5, 7.7); \draw plot[mark=*,
      mark options={fill=white}] (7.5, 12); \draw plot[mark=, mark
      options={fill=black}] (8.5, 14); \draw plot[mark=*, mark
      options={fill=black}] (2, 26); \draw plot[mark=*, mark
      options={fill=black}] (8.5, 7.7);
    \end{tikzpicture}
    \caption{Structure of the matrix associated with $\Psi_H$.}\label{fig2}
  \end{figure}

  Let $T_4 = \Psi_H T_3$. Then, for $f\in C([0,\omega_1])$ and
  $\alpha\in[0,\omega_1]$, we have
  \begin{equation}\label{matrixT4new} (T_4f)(\alpha) =
    (T_3f)(\psi_H(\alpha)) = 
    \sum_{\beta\in[0,\omega_1]} (T_3)_{\psi_H(\alpha),\beta}f(\beta) =
    (T_3)_{\psi_H(\alpha),\gamma} f(\gamma) \end{equation}
  by~\eqref{supprT3}, where $\gamma =(\theta\circ\psi_H)(\alpha)$ and
  $\psi_H\colon[0,\omega_1]\to H$ denotes the order iso\-mor\-phism as
  in Lemma~\ref{anothercompop}.  Taking $\alpha = \omega_1$ and $f =
  \mathbf{1}_{[0,\omega_1]}$, we obtain
  $T_4(\mathbf{1}_{[0,\omega_1]})(\omega_1) =
  (T_3)_{\omega_1,\omega_1} = 1$.  
  Being continuous,  the function $T_4(\mathbf{1}_{[0,\omega_1]})$ is
  eventually constant,  so we can find a
  countable ordinal~$\chi$ such that 
  $T_4(\mathbf{1}_{[0,\omega_1]})(\alpha) = 1$ for each
  $\alpha\in[\chi,\omega_1]$. Moreover, \eqref{matrixT4new} implies
  that $\supp(r_\alpha^{T_4})\subseteq\{(\theta\circ\psi_H)(\alpha)\}$
  for each  $\alpha\in[0,\omega_1]$, hence we conclude  that
  \begin{equation}\label{alpharowofT4}
    (T_4)_{\alpha,\beta} =
    \delta_{(\theta\circ\psi_H)(\alpha),\beta}\qquad
    (\alpha\in[\chi,\omega_1],\,\beta\in[0,\omega_1]). \end{equation}
 
  Let $T_5 = QT_4$, where $Q = I-P_\chi +
  \mathbf{1}_{[0,\chi]}\otimes\epsilon_\chi$. An easy computation
  gives
  \[ Q_{\alpha,\beta} = \begin{cases} \delta_{\chi,\beta}
    &\text{for}\ \alpha\in[0,\chi]\\
    \delta_{\alpha,\beta} &\text{for}\
    \alpha\in(\chi,\omega_1] \end{cases}\qquad
  (\alpha,\beta\in[0,\omega_1]), \] which together
  with~\eqref{alpharowofT4} implies that
  \begin{equation}\label{T5matrix} (T_5)_{\alpha,\gamma} =
    \sum_{\beta\in[0,\omega_1]}
    Q_{\alpha,\beta}(T_4)_{\beta,\gamma} = \begin{cases}
      \delta_{(\theta\circ\psi_H)(\chi),\gamma} &\text{for}\
      \alpha\in[0,\chi]\\
      \delta_{(\theta\circ\psi_H)(\alpha),\gamma} &\text{for}\
      \alpha\in(\chi,\omega_1] \end{cases}\qquad
    (\alpha,\gamma\in[0,\omega_1]).  \end{equation}
  This shows in particular that $k_{\omega_1}^{T_5} =
  \mathbf{1}_{\{\omega_1\}}$, so 
  $T_5\notin\mathscr{M}$, and consequently
  the set \[ \Gamma =
  \{\gamma\in[0,\omega_1] :   k_\gamma^{T_5}\neq 0\} =
  (\theta\circ\psi_H)([\chi,\omega_1]) \] is   uncountable. Let $M =  
  (\mu_\sigma)_{\sigma\in[0,\omega_1]}$ be the increasing enumeration
  of~$\Gamma$.  We note that $\mu_0 = (\theta\circ\psi_H)(\chi)$ and
  $\mu_{\omega_1} = \omega_1$, and for each $\sigma\in[0,\omega_1]$,
  we have 
  \begin{equation}\label{eqT5indifunction}
    T_5(\mathbf{1}_{[0,\mu_\sigma]}) = 
    \mathbf{1}_{[0,\nu_\sigma]},\qquad \text{where}\qquad \nu_\sigma =
    \sup\{\alpha\in[0,\omega_1] :
    (\theta\circ\psi_H)(\alpha)\le\mu_\sigma\}. \end{equation} 
  The transfinite sequence $N =
  (\nu_\sigma)_{\sigma\in[0,\omega_1]}$ is  clearly increasing; 
  to see that it increases strictly, suppose that
  $0\le\tau<\sigma\le\omega_1$. Then $\mu_\tau<\mu_\sigma$. 
  On the one hand, since $\mu_\sigma\in\Gamma$, we have $\mu_\sigma =
  (\theta\circ\psi_H)(\alpha)$ for some   $\alpha\in[\chi,\omega_1]$,
  and therefore 
  \[ T_5(\mathbf{1}_{[\mu_\tau+1,\mu_\sigma]})(\alpha) =
  \sum_{\gamma\in [\mu_\tau+1,\mu_\sigma]} (T_5)_{\alpha,\gamma} =
  1 \] by \eqref{T5matrix}.  On the other,
  \eqref{eqT5indifunction}~implies that
  $T_5(\mathbf{1}_{[\mu_\tau+1,\mu_\sigma]}) =
  \mathbf{1}_{[0,\nu_\sigma]} - \mathbf{1}_{[0,\nu_\tau]}$. The only
  way that this function can take the value~$1$ at~$\alpha$ is if
  $\nu_\tau<\alpha\le\nu_\sigma$, and the conclusion follows.

  Lemma~\ref{thePhiLemma}\romanref{thePhiLemma1} implies that there
  are linear isometries $U_M$ and $U_N$ on~$C([0,\omega_1])$ such that
  $U_M(\mathbf{1}_{[0,\sigma]}) = \mathbf{1}_{[0,\mu_\sigma]}$ and
  $U_N(\mathbf{1}_{[0,\sigma]}) = \mathbf{1}_{[0,\nu_\sigma]}$ for
  each $\sigma\in[0,\omega_1]$. Moreover, their ranges are
  complemented in~$C([0,\omega_1])$ by
  Lemma~\ref{thePhiLemma}\romanref{thePhiLemma3}; the projection onto
  $U_N(C([0,\omega_1]))$ is~$\Phi_N$, and so we can define an operator
  $V_N = U_N^{-1}\Phi_N$ on $C([0,\omega_1])$.  We now see that $V_N
  T_5 U_M = I$ because $V_N T_5 U_M(\mathbf{1}_{[0,\sigma]}) =
  \mathbf{1}_{[0,\sigma]}$ for each $\sigma\in[0,\omega_1]$, and the
  result follows.
\end{proof}

\begin{remark} Ogden \cite{Og} extended the definition of
  $\mathscr{M}$ to the case of $\mathscr{B}(C([0,\omega_\eta]))$,
  where $\eta$ is any ordinal such that $\omega_\eta$ is a regular
  cardinal. Our main result is valid also in this case with a similar
  proof. \end{remark}

\section{Operators with separable range: the proof of
  Theorem~\ref{thmseprange}}\label{secion4}

\noindent
We require four lemmas. The first is straightforward, so we omit its
proof.

\begin{lemma}\label{lemma28june1}
  Let $K$ be a compact topological space, and let
  $(f_\alpha)_{\alpha\in[0,\omega_1)}$ be a family of pairwise
  disjointly supported func\-tions in~$C(K)$ such that
  $\sup\{\|f_\alpha\| : \alpha\in[0,\omega_1)\} < \infty$ and
  $\inf\{\|f_\alpha\| : \alpha\in[0,\omega_1)\} > 0$. Then
  $(f_\alpha)_{\alpha\in[0,\omega_1)}$ is a transfinite basic sequence
  equivalent to the canonical Schauder basis
  $(\mathbf{1}_{\{\alpha\}})_{\alpha\in[0,\omega_1)}$
  for~$c_0(\omega_1)$.
\end{lemma}

\begin{lemma}\label{lemma12sept}
  A subspace $X$ of~$C([0,\omega_1])$ is separable if and only if $X$
  is contained in the range of the projection~$\widetilde{P}_\sigma$
  given by~\eqref{defnPtilde} for some countable ordinal~$\sigma$.
\end{lemma}
\begin{proof}
  The implication $\Leftarrow$ is clear.  Conversely, suppose that $W$
  is a countable dense subset of~$X$.  Since each continuous function
  on~$[0,\omega_1]$ is eventually constant, we can choose a countable
  ordinal~$\sigma$ such that $f|_{[\sigma+1,\omega_1]}$ is constant
  for each $f\in W$. This implies that $\widetilde{P}_\sigma f = f$
  for each $f\in W$, so as $\widetilde{P}_\sigma$ has closed range and
  $W$ is dense in~$X$, we conclude that $X\subseteq
  \widetilde{P}_\sigma(C([0,\omega_1]))$.
\end{proof}

\begin{lemma}\label{lemma28june2}
  Let $T$ be an operator on~$C([0,\omega_1])$ such that
  $T\ne\widetilde{P}_\sigma T\widetilde{P}_\sigma$ for each countable
  ordinal~$\sigma$.  Then there is an $\epsilon>0$ such that, for each
  countable ordinal~$\xi$, there is a function $f\in C([0,\omega_1])$
  with $\supp(f)\subseteq (\xi,\omega_1)$ satisfying $\|f\|\le 1$ and
  $\|Tf\|\ge\epsilon$.
\end{lemma}

\begin{proof}[Proof by contraposition] Suppose that the conclusion is
  false. Then, taking \mbox{$\epsilon = 1/n$} for $n\in\mathbb{N}$, we
  obtain a sequence $(\xi_n)_{n\in\mathbb{N}}$ of countable ordinals
  such that \mbox{$\|Tf\|<1/n$} for each function $f\in
  C([0,\omega_1])$ with $\supp(f)\subseteq (\xi_n,\omega_1)$ and $\|
  f\|\le 1$.  

  We claim that the countable ordinal \mbox{$\xi = \sup\{\xi_n :
    n\in\mathbb{N}\}$} satisfies $T=T\widetilde{P}_\xi$. To verify
  this claim, it clearly suffices to prove that
  $T(I-\widetilde{P}_\xi)g = 0$ for each $g\in C([0,\omega_1])$ with
  $\|(I-\widetilde{P}_\xi)g\|\le 1$.  Letting $f =
  (I-\widetilde{P}_\xi)g$, we have $\supp(f)\subseteq (\xi,\omega_1) =
  \bigcap_{n\in\mathbb{N}} (\xi_n,\omega_1)$ because
  $P_\xi(I-\widetilde{P}_\xi) = 0$ and $f(\omega_1) = 0$. Hence the
  choice of~$\xi_n$ implies that $\|Tf\|< 1/n$ for each
  $n\in\mathbb{N}$, so $0 = Tf = T(I-\widetilde{P}_\xi)g$, and the
  claim follows.

  In particular, $T$ has separable range, so Lemma~\ref{lemma12sept}
  implies that $T = \widetilde{P}_\eta T$ for some countable
  ordinal~$\eta$.  Since $\widetilde{P}_\alpha\widetilde{P}_\beta =
  \widetilde{P}_{\min\{\alpha,\beta\}}$, we conclude that $T =
  \widetilde{P}_\sigma T\widetilde{P}_\sigma$ is satisfied for $\sigma
  = \max\{\xi,\eta\}$.
\end{proof}

\begin{lemma}\label{lemma28june3}
  Let $S$ be an operator on~$C([0,\omega_1])$ with $k_{\omega_1}^S =
  0$. For each pair $\zeta,\eta$ of countable ordinals, there is a
  countable ordinal $\xi\ge\zeta$ such that $P_\eta S(I-P_\xi) = 0$.
\end{lemma}
\begin{proof}
  Let $\xi = \sup\bigl(\{\zeta\}\cup\bigcup_{\alpha\in[0,\eta]}\supp
  (r^S_\alpha)\bigr)$. Then clearly $\zeta\le\xi$, and $\xi$ is
  countable because $\supp(r^S_\alpha)$ is countable and
  $S_{\alpha,\omega_1} = 0$ for each~$\alpha$. We show that
  \mbox{$P_\eta S(I-P_\xi) = 0$} by verifying that $(P_\eta
  S(I-P_\xi))_{\alpha,\delta} = 0$ for each pair $\alpha,
  \delta\in[0,\omega_1]$. Indeed, by~\eqref{matrixmult}, we have
  \[ (P_\eta S(I-P_\xi))_{\alpha,\delta} =
  \sum_{\beta,\gamma\in[0,\omega_1]}
  (P_\eta)_{\alpha,\beta}S_{\beta,\gamma} (I -
  P_\xi)_{\gamma,\delta}= \begin{cases} 0 &\text{if}\
    \alpha\in(\eta,\omega_1],\\
    0 &\text{if}\ \delta\in[0,\xi],\\
    S_{\alpha,\delta} &\text{otherwise,} \end{cases} \] and
  $S_{\alpha,\delta} = 0$ for $\alpha\in[0,\eta]$ and
  $\delta\in(\xi,\omega_1]$ by the choice of~$\xi$.
\end{proof}

\begin{proof}[Proof of Theorem~\ref{thmseprange}.] The implications
  \alphref{thmseprange2e}$\Rightarrow$\alphref{thmseprange2g}%
  $\Rightarrow$\alphref{thmseprange2f}$\Rightarrow$%
  \alphref{thmseprange2a}$\Rightarrow$\alphref{thmseprange2d} are all
  straight\-forward. Indeed,
  \alphref{thmseprange2e}$\Rightarrow$\alphref{thmseprange2g} because
  $\widetilde{P}_\sigma$ is a rank-one perturbation of~$P_\sigma$,
  whose range is iso\-metrically isomorphic to~$C([0,\sigma])$;
  \alphref{thmseprange2g}$\Rightarrow$\alphref{thmseprange2f} is
  obvious; \alphref{thmseprange2f}$\Rightarrow$\alphref{thmseprange2a}
  follows from the facts that~$C([0,\sigma])$ is separable and
  $\mathscr{X}(C([0,\omega_1]))$ is a closed operator ideal; and
  \alphref{thmseprange2a}$\Rightarrow$\alphref{thmseprange2d}
  because~$c_0(\omega_1)$ is non-separable.
 
  Finally, we prove that
  \alphref{thmseprange2d}$\Rightarrow$\alphref{thmseprange2e} by
  contraposition. Suppose that $T\ne \widetilde{P}_\sigma
  T\widetilde{P}_\sigma$ for each countable ordinal~$\sigma$.  If
  $T\notin\mathscr{M}$, then Theorem~\ref{coorfreecharofLWideal}
  implies that $T$ fixes a copy of~$C([0,\omega_1])$ and thus
  of~$c_0(\omega_1)$. Otherwise choose $\epsilon>0$ as in
  Lemma~\ref{lemma28june2}. By in\-duc\-tion, we shall construct a
  family $(f_\alpha)_{\alpha\in[0,\omega_1)}$ of functions in
  $C([0,\omega_1])$ such that \mbox{$\sup\{\|f_\alpha\| :
    \alpha\in[0,\omega_1)\}\le 1$}, \mbox{$\inf\{\|Tf_\alpha\| :
    \alpha\in[0,\omega_1)\}\ge\epsilon$}, $f_0(\omega_1) = 0$,
  $Tf_0(\omega_1)=0$ and
  \begin{equation}\label{eq1June29}
    \supp(f_\alpha)\subseteq (\sup\supp(f_\beta),\omega_1)\qquad
    \text{and}\qquad \supp(Tf_\alpha)\subseteq (\sup\supp
    (Tf_\beta),\omega_1) \end{equation}
  whenever $0\le\beta<\alpha<\omega_1$. Before  giving the details of
  this construction, let us explain how it enables us to complete the
  proof. The families $(f_\alpha)_{\alpha\in[0,\omega_1)}$ and
  $(Tf_\alpha)_{\alpha\in[0,\omega_1)}$ both satisfy the conditions in 
  Lemma~\ref{lemma28june1}, so they are equivalent to the  canonical
  Schauder basis for~$c_0(\omega_1)$. Hence, as $T$ maps
  $(f_\alpha)_{\alpha\in[0,\omega_1)}$
  on\-to~$(Tf_\alpha)_{\alpha\in[0,\omega_1)}$, it fixes a copy
  of~$c_0(\omega_1)$.  
 
  It remains to inductively construct
  $(f_\alpha)_{\alpha\in[0,\omega_1)}$. To start the in\-duc\-tion, we
  note that $\xi = \sup(\supp(r_{\omega_1}^T)\setminus\{\omega_1\})$
  is a countable ordinal by
  Propo\-si\-tion~\ref{LoyWillis3.1}\romanref{LoyWillis3.1i}.  Lemma
  \ref{lemma28june2} there\-fore enables us to choose a function
  $f_0\in C([0,\omega_1])$ with \mbox{$\supp(f_0)\subseteq
    (\xi,\omega_1)$} such that $\|f_0\|\le 1$ and
  $\|Tf_0\|\ge\epsilon$. Of the conditions that~$f_0$ must satisfy,
  only $Tf_0(\omega_1)=0$ is not evident; however, we have
  \begin{equation}\label{proof1.3eq} Tf_0(\omega_1) =
    \sum_{\beta\in[0,\omega_1]} T_{\omega_1,\beta}f_0(\beta) = 0 \end{equation}
  because $f_0(\beta) = 0$ for
  $\beta\in[0,\xi]\cup\{\omega_1\}$, while $T_{\omega_1,\beta} = 0$ for
  $\beta\in(\xi,\omega_1)$ by the choice of~$\xi$.

  Now let $\alpha\in(0,\omega_1)$, and assume inductively that
  functions $(f_\beta)_{\beta\in[0,\alpha)}$ in $C([0,\omega_1])$ have
  been chosen as specified. The func\-tion~$k^T_{\omega_1}$ is
  continuous because $T\in\mathscr{M}$, so we may define a rank-one
  operator by $F = k^T_{\omega_1}\otimes \epsilon_{\omega_1}$. Since
  $k^{T-F}_{\omega_1} = 0$, we can apply Lemma~\ref{lemma28june3} with
  \[ \zeta = \sup\biggl((\supp
  (r_{\omega_1}^T)\setminus\{\omega_1\})\cup
  \bigcup_{\beta\in[0,\alpha)}\supp(f_\beta)\biggr)\qquad
  \text{and}\qquad \eta =
  \sup\biggl(\bigcup_{\beta\in[0,\alpha)}\supp(Tf_\beta)\biggr) \] to
  obtain a countable ordinal $\xi\ge\zeta$ such that
  $P_\eta(T-F)(I-P_\xi) = 0$. (Note that the ordinals $\zeta$ and
  $\eta$ are countable because $f_\beta$ and $Tf_\beta$ are continuous
  functions on~$[0,\omega_1]$ mapping~$\omega_1$ to~$0$, so they have
  countable supports for each $\beta\in[0,\alpha)$.) By
  Lemma~\ref{lemma28june2}, we can take a function $f_\alpha\in
  C([0,\omega_1])$ with $\supp(f_\alpha)\subseteq (\xi,\omega_1)$ such
  that $\|f_\alpha\|\le 1$ and $\|Tf_\alpha\|\ge\epsilon$. It remains
  to check that~\eqref{eq1June29} holds for each
  $\beta\in[0,\alpha)$. The first statement is clear because
  $\supp(f_\alpha)\subseteq (\xi,\omega_1)$ and
  $\sup\supp(f_\beta)\le\zeta\le\xi$. To verify the second, we observe
  that $Tf_\alpha(\omega_1) = 0$ by an argument similar to that given
  in~\eqref{proof1.3eq} above. Moreover, since $f_\alpha\in\ker P_\xi$
  and $f_\alpha\in\ker\epsilon_{\omega_1} = \ker F$, we have
  \[ P_\eta Tf_\alpha = P_\eta (T-F)(I-P_\xi)f_\alpha = 0. \]
  Consequently, $\supp(Tf_\alpha)\subseteq (\eta,\omega_1)$, from
  which the desired conclusion follows because
  $\sup\supp(Tf_\beta)\le\eta$. Hence the induction continues.
\end{proof}

\section{The lattice of closed ideals of
  $\mathscr{B}(C([0,\omega_1]))$}\label{section5ideallattice}
\noindent
The aim of this section is to establish the hierarchy among the closed
ideals of $\mathscr{B}(C([0,\omega_1]))$ shown in
Figure~\ref{diagramClosedIdeals}.  Beginning from the bottom of the
diagram, we note that as $C([0,\omega_1])$ is a
$\mathscr{L}_\infty$-space, it has the bounded approximation property,
so $\mathscr{K}(C([0,\omega_1]))$ is the closure of the ideal of
finite-rank operators and thus the minimum non-zero closed ideal.

To prove the minimality of the next two inclusions in
Figure~\ref{diagramClosedIdeals}, we require the following variant of
Sobczyk's theorem for~$C([0,\omega_1])$, which is due to Argyros
\emph{et al}.
\begin{proposition}[{\cite[Proposition~3.2]{argyrosetal}}]%
\label{argyrosetalProp3.2}
Let $X$ be a subspace of $C([0,\omega_1])$ which is isomorphic to
either~$c_0$ or~$c_0(\omega_1)$. Then $X$ is automatically
complemented.
\end{proposition}

\begin{remark}
  The first part of Proposition~\ref{argyrosetalProp3.2} follows
  easily from our results and Sobczyk's theorem. Indeed, let $X$ be a
  subspace of $C([0,\omega_1])$ which is isomorphic to~$c_0$. Then $X$
  is separable, hence contained in
  $\widetilde{P}_\sigma(C([0,\omega_1]))$ for some countable
  ordinal~$\sigma$ by Lemma~\ref{lemma12sept}.  Sobczyk's theorem
  implies that $X$ is complemented in
  $\widetilde{P}_\sigma(C([0,\omega_1]))$, and as
  $\widetilde{P}_\sigma(C([0,\omega_1]))$ is complemented
  in~$C([0,\omega_1])$, so is~$X$.
\end{remark}

\begin{proposition}
  The identity operator on~$c_0$ factors through each non-compact
  operator on~$C([0,\omega_1])$.  Hence no closed ideal of
  $\mathscr{B}(C([0,\omega_1]))$ lies strictly
  between~$\mathscr{K}(C([0,\omega_1]))$ and
  $\overline{\mathscr{G}}_{c_0}(C([0,\omega_1]))$.
\end{proposition}
\begin{proof} 
  This is a standard argument which we outline for
  completeness. Since $[0,\omega_1]$ is scattered,
  $C([0,\omega_1])^*\cong \ell_1([0,\omega_1])$, so
  $C([0,\omega_1])^*$ has the Schur property. Hence all weakly compact
  operators on~$C([0,\omega_1])^*$ are compact. The theorems of Gantmacher and
  Schauder then imply that all weakly compact operators
  on~$C([0,\omega_1])$ are compact, and therefore, by a theorem of
  Pe{\l}czy{\a'n}ski, each non-compact operator on~$C([0,\omega_1])$
  fixes a copy of~$c_0$. Now the conclusion follows from
  Proposition~\ref{argyrosetalProp3.2}.
\end{proof}

For each countable ordinal~$\alpha$, let $Q_\alpha$ denote the
$\alpha^{\text{th}}$ projection associated with the canonical Schauder
basis $(\mathbf{1}_{\{\beta\}})_{\beta\in[0,\omega_1)}$
for~$c_0(\omega_1)$; that is, $(Q_\alpha f)(\beta) = f(\beta)$ for
$\beta\in[0,\alpha]$ and $(Q_\alpha f)(\beta) = 0$ for
$\beta\in(\alpha,\omega_1)$. We can use the projections~$Q_\alpha$ to
characterize the separable subspaces of~$c_0(\omega_1)$ in a similar
fashion to Lemma~\ref{lemma12sept} for $C([0,\omega_1])$.  Although
this characterization follows easily from standard results such
as~\cite[Proposition~5.6]{hmvz}, we outline a short, elementary proof.

\begin{lemma}\label{sepsubspaceofc0omega1}
  A subspace $X$ of~$c_0(\omega_1)$ is separable if and only if $X$ is
  contained in the range of the projection $Q_\alpha$ for some
  countable ordinal~$\alpha$.
\end{lemma}

\begin{proof}
The implication $\Leftarrow$ is immediate because $Q_\alpha$ has
separable range for each $\alpha\in[0,\omega_1)$. 

Conversely, suppose that $X$ is separable, and let $W$ be a dense,
countable subset of~$X$. Since each element of~$c_0(\omega_1)$ has
countable support, the ordinal $\alpha = \sup\bigcup_{f\in W}\supp f$
is countable, and clearly $Q_\alpha f = f$ for each $f\in W$.  Hence
$W$ is contained in the range of~$Q_\alpha$, which is closed, so the
same is true for~$X$.
\end{proof}

\begin{proposition} 
  No closed ideal of $\mathscr{B}(C([0,\omega_1]))$ lies strictly
  between $\overline{\mathscr{G}}_{c_0}(C([0,\omega_1]))$ and
  $\overline{\mathscr{G}}_{c_0(\omega_1)}(C([0,\omega_1]))$.
\end{proposition} 
\begin{proof}
  Given $T\in\overline{\mathscr{G}}_{c_0(\omega_1)}(C([0,\omega_1]))$,
  we consider two cases.  If $T\in\mathscr{X}(C([0,\omega_1]))$, then
  Theorem~\ref{thmseprange} shows that $T = \widetilde{P}_\sigma
  T\widetilde{P}_\sigma$ for some countable ordinal~$\sigma$. Given
  $\epsilon>0$, choose operators $U\colon C([0,\omega_1])\to
  c_0(\omega_1)$ and $V\colon c_0(\omega_1)\to C([0,\omega_1])$ such
  that $\|T - VU\|\le\epsilon$. Since the range of the operator
  $U\widetilde{P}_\sigma$ is separable, we can take a countable
  ordinal~$\alpha$ such that $Q_\alpha U\widetilde{P}_\sigma =
  U\widetilde{P}_\sigma$ by Lemma~\ref{sepsubspaceofc0omega1}. Hence
  we have
  \[ \| T - VQ_\alpha U\widetilde{P}_\sigma\| = \|
  T\widetilde{P}_\sigma - VU\widetilde{P}_\sigma\|\le \| T- VU\|\,
  \|\widetilde{P}_\sigma\|\le \epsilon, \] so
  $T\in\overline{\mathscr{G}}_{c_0}(C([0,\omega_1]))$ because
  $Q_\alpha\in\mathscr{G}_{c_0}(c_0(\omega_1))$.

  Otherwise $T\notin\mathscr{X}(C([0,\omega_1]))$, and
  Theorem~\ref{thmseprange} implies that $T$ fixes a copy~$X$
  of~$c_0(\omega_1)$. Proposition~\ref{argyrosetalProp3.2} ensures
  that $T(X)$ is complemented in~$C([0,\omega_1])$, so the closed
  ideal generated by~$T$ is equal
  to~$\overline{\mathscr{G}}_{c_0(\omega_1)}(C([0,\omega_1]))$.
\end{proof}

To complete Figure~\ref{diagramClosedIdeals}, we require the
\emph{Szlenk index} as it enables us to dis\-tin\-guish the
$C(K)$-spaces considered therein.  This ordinal-valued index, denoted
by~$\Sz X$, was originally intro\-duced by Szlenk~\cite{szlenk} for
Banach spaces~$X$ with separable dual and has subsequently been
generalized to encompass all Asp\-lund spaces (or all Banach spaces,
provided that one is willing to accept that $\Sz X$ takes the value
`un\-defined' (or~$\infty$) if $X$ is not an Asp\-lund space).  We
shall not state the definition of the Szlenk index here as all we need
to know is its value for certain $C(K)$-spaces. The interested reader
is referred to~\cite[Section~2.4]{hmvz} for a modern introduction to
the Szlenk index. 

A proof of the first part of the following theorem is outlined
in~\cite[Exercise~8.55]{fhhmz}, while the second, much deeper, part is
due to Samuel~\cite{samuel}; a simplified proof of it, due to
H\'{a}jek and Lancien, can be found in~\cite{hl}
or~\cite[Theorem~2.59]{hmvz}.
\begin{theorem}\label{SzlenkCK}
  \begin{romanenumerate}
  \item\label{SzlenkCK2} The Szlenk index of~$c_0(\omega_1)$
    is~$\omega$.
  \item\label{SzlenkCK1} 
 Let $\alpha$
    be a countable ordinal. Then $C([0,\omega^{\omega^\alpha}])$ has
    Szlenk index~$\omega^{\alpha+1}$.
  \end{romanenumerate}
\end{theorem}

In fact, a Szlenk index can be associated with each operator between
Banach spaces in such a way that the Szlenk index of a Banach space is
equal to that of its identity operator.  We are interested in this
notion because Brooker~\cite[Theorem~2.2]{brooker} has shown that, for
each ordinal~$\alpha$, the
collection~$\mathscr{S}\!\mathscr{Z}_\alpha$ of operators having
Szlenk index at most~$\omega^\alpha$ forms a closed operator ideal in
the sense of Pietsch.

Armed with this information, we can prove that all inclusions are
proper in each of the two infinite ascending chains in
Figure~\ref{diagramClosedIdeals}.

\begin{proposition}\label{propinfinitechains}
  Let $\alpha$ be a countable ordinal, and let
  $K_\alpha = [0,\omega^{\omega^\alpha}]$. Then:
  \begin{romanenumerate}
  \item\label{propinfinitechains1}
    $\overline{\mathscr{G}}_{C(K_{\alpha})}(C([0,\omega_1]))\subsetneq
    \overline{\mathscr{G}}_{C(K_\alpha)\oplus
      c_0(\omega_1)}(C([0,\omega_1]));$
  \item\label{propinfinitechains2}
    $\overline{\mathscr{G}}_{C(K_{\alpha})}(C([0,\omega_1]))\subsetneq
    \overline{\mathscr{G}}_{C(K_{\alpha+1})}(C([0,\omega_1]));$
  \item\label{propinfinitechains3}
    $\overline{\mathscr{G}}_{C(K_{\alpha})\oplus
      c_0(\omega_1)}(C([0,\omega_1]))\subsetneq
    \overline{\mathscr{G}}_{C(K_{\alpha+1})\oplus
      c_0(\omega_1)}(C([0,\omega_1]));$
  \item\label{propinfinitechains4}
    $\overline{\mathscr{G}}_{c_0(\omega_1)}(C([0,\omega_1]))\subsetneq
    \overline{\mathscr{G}}_{C(K_1)\oplus
      c_0(\omega_1)}(C([0,\omega_1]))$.
  \end{romanenumerate}
\end{proposition}

\begin{proof}
  To prove~\romanref{propinfinitechains1}, let~$Q$ be a projection on
  $C([0,\omega_1])$ whose range is isomorphic to~$c_0(\omega_1)$. Then
  $Q\in \overline{\mathscr{G}}_{C(K_\alpha)\oplus
    c_0(\omega_1)}(C([0,\omega_1]))$, but $Q\notin
  \overline{\mathscr{G}}_{C(K_{\alpha})}(C([0,\omega_1]))$ because its
  range is non-separable.

  We shall prove~\romanref{propinfinitechains2}
  and~\romanref{propinfinitechains3} simultaneously by displaying an
  operator belonging to
  $\overline{\mathscr{G}}_{C(K_{\alpha+1})}(C([0,\omega_1]))%
  \setminus \overline{\mathscr{G}}_{C(K_{\alpha})\oplus
    c_0(\omega_1)}(C([0,\omega_1]))$.  More precisely, we claim that
  the pro\-jec\-tion~$P_\sigma$ is such an operator for $\sigma =
  \omega^{\omega^{\alpha+1}}$.  Indeed,
  $P_\sigma\in\overline{\mathscr{G}}_{C(K_{\alpha+1})}%
  (C([0,\omega_1]))$ because its range is isometrically isomorphic
  to~$C(K_{\alpha+1})$.  On the other hand,
  Theorem~\ref{SzlenkCK}\romanref{SzlenkCK1} implies that the identity
  operator on~$C(K_{\alpha+1})$ does not belong to the operator ideal
  $\mathscr{S}\!\mathscr{Z}_{\alpha+1}$. Since it factors
  through~$P_\sigma$, we deduce that
  $P_\sigma\notin\mathscr{S}\!\mathscr{Z}_{\alpha+1}(C([0,\omega_1]))$,
  and consequently we have $P_\sigma\notin
  \overline{\mathscr{G}}_{C(K_{\alpha})\oplus
    c_0(\omega_1)}(C([0,\omega_1]))$ because
  $\overline{\mathscr{G}}_{C(K_{\alpha})\oplus
    c_0(\omega_1)}\subseteq\mathscr{S}\!\mathscr{Z}_{\alpha+1}$ as the
  following calculation shows
  \[ \Sz C(K_\alpha)\oplus c_0(\omega_1) = \max\{\Sz C(K_\alpha), \Sz
  c_0(\omega_1)\} = \omega^{\alpha+1}. \] Here, the first equality
  follows from~\cite[Proposition~1.5(v)]{brooker} (which in turn is a
  consequence of~\cite[Equation~(2.3)]{hl}) and the second from
  Theorem~\ref{SzlenkCK}.

  Finally, \romanref{propinfinitechains4} follows by taking $\alpha =
  0$ in~\romanref{propinfinitechains3} because $C(K_0) =
  C([0,\omega])\cong c_0$, so that $C(K_0)\oplus c_0(\omega_1)\cong
  c_0(\omega_1)$.
\end{proof}

We now come to the most interesting result in this section.

\begin{theorem}\label{XplusGandM}
  The ideal $\mathscr{X}(C([0,\omega_1])) +
  \overline{\mathscr{G}}_{c_0(\omega_1)}(C([0,\omega_1]))$ is closed, 
  and 
  \[ \mathscr{X}(C([0,\omega_1]))\subsetneq
  \mathscr{X}(C([0,\omega_1])) +
  \overline{\mathscr{G}}_{c_0(\omega_1)}(C([0,\omega_1]))%
  \subsetneq\mathscr{M}. \]
\end{theorem}
\begin{proof}
  To show that the ideal $\mathscr{X}(C([0,\omega_1])) +
  \overline{\mathscr{G}}_{c_0(\omega_1)}(C([0,\omega_1]))$ is closed,
  let $T$ be an operator belonging to its closure, and take sequences
  $(R_n)_{n\in\mathbb{N}}$ in $\mathscr{X}(C([0,\omega_1]))$ and
  $(S_n)_{n\in\mathbb{N}}$ in
  $\overline{\mathscr{G}}_{c_0(\omega_1)}(C([0,\omega_1]))$ such that
  $R_n+S_n\to T$ as $n\to\omega$. Then $\lin\bigcup_{n\in\mathbb{N}}
  R_n(C([0,\omega_1]))$ is a separable subspace of~$C([0,\omega_1])$,
  so Lemma~\ref{lemma12sept} implies that it is contained in the range
  of~$\widetilde{P}_\sigma$ for some countable ordinal~$\sigma$. Hence
  we have $\widetilde{P}_\sigma R_n = R_n$ for each $n\in\mathbb{N}$,
  and therefore
  \[ (I - \widetilde{P}_\sigma)S_n = (I - \widetilde{P}_\sigma)(R_n +
  S_n)\to (I - \widetilde{P}_\sigma)T\qquad \text{as}\qquad
  n\to\omega, \] so $(I -
  \widetilde{P}_\sigma)T\in\overline{\mathscr{G}}_{c_0(\omega_1)}%
  (C([0,\omega_1]))$. Since $\widetilde{P}_\sigma$ has separable
  range, we conclude that \[ T = \widetilde{P}_\sigma T + (I -
  \widetilde{P}_\sigma)T\in \mathscr{X}(C([0,\omega_1])) +
  \overline{\mathscr{G}}_{c_0(\omega_1)}(C([0,\omega_1])), \] as
  required.

  We have $\mathscr{X}(C([0,\omega_1]))\subsetneq
  \mathscr{X}(C([0,\omega_1])) +
  \overline{\mathscr{G}}_{c_0(\omega_1)}(C([0,\omega_1]))$ because
  $C([0,\omega_1])$ contains a complemented subspace isomorphic
  to~$c_0(\omega_1)$, which is non-separable.

  The proof that $\mathscr{X}(C([0,\omega_1])) +
  \overline{\mathscr{G}}_{c_0(\omega_1)}(C([0,\omega_1]))$ is properly
  contained in the Loy--Willis ideal~$\mathscr{M}$ is somewhat more
  involved. By Theorem~\ref{uniquenessthm}, it suffices to display an
  operator belonging to the latter, but not the former ideal. We
  construct such an operator by considering an operator whose range is
  contained in a certain $C(K)$-subspace of~$C([0,\omega_1])$.

  Let $H = \{\omega^\lambda : \lambda\in[\omega,\omega_1]\ \text{is a
    limit ordinal}\}$; note that $\omega_1\in H$ because
  $\omega^{\omega_1} = \omega_1$. We define an equivalence
  rela\-tion~$\sim$ on~$[0,\omega_1]$ by
  \[ \alpha\sim \beta\qquad \Longleftrightarrow\qquad (\alpha =
  \beta\quad \text{or}\quad \alpha, \beta\in H)\qquad\quad
  (\alpha,\beta\in [0,\omega_1]). \] Denote by~$K$ the quotient space
  $[0,\omega_1]/\!\!\sim$ equipped with the quotient topo\-logy, and
  let $\pi\colon[0,\omega_1]\to K$ be the quotient map.  Then $K$ is
  compact (as the continuous image of a compact space), and the
  composition operator \mbox{$C_\pi\colon f\mapsto f\circ\pi$},
  \mbox{$C(K)\to C([0,\omega_1])$}, is a linear isometry.

  For each $\alpha\in[0,\omega_1]\setminus H$, either
  $\alpha\in[0,\omega^\omega)$ or
  $\alpha\in(\omega^\lambda,\omega^{\lambda+\omega})$ for some limit
  ordinal $\lambda\in[\omega,\omega_1)$. In the first case, let
  $A_\alpha = [0,\alpha]$, in the second, let $A_\alpha =
  (\omega^\lambda,\alpha]$. Then $A_\alpha$ is clopen
  in~$[0,\omega_1]$ and disjoint from~$H$, so $\pi(A_\alpha)$ is
  clopen in~$K$.
  
  This implies that $K$ is Hausdorff. Indeed, given two distinct
  points $\pi(\alpha),\pi(\beta)\in K$, we may suppose that
  $\alpha\notin H$ and $\alpha<\beta$. Then $\pi(A_\alpha)$ and
  $K\setminus\pi(A_\alpha)$ are disjoint open neigh\-bour\-hoods
  of~$\pi(\alpha)$ and~$\pi(\beta)$, respectively.

  Moreover, we can define a linear map $U\colon
  \lin\{\mathbf{1}_{[0,\alpha]}:\alpha\in[0,\omega_1]\}\to C(K)$ by
  \begin{equation}\label{defnU} U\mathbf{1}_{[0,\alpha]}
    = \begin{cases} \mathbf{1}_{\pi(A_\alpha)}
      &\text{for}\ \alpha\in[0,\omega_1]\setminus H\\ 
      0 &\text{for}\ \alpha\in H
    \end{cases} \end{equation} because $\pi(A_\alpha)$ being clopen
  ensures that the indicator function~$\mathbf{1}_{\pi(A_\alpha)}$ is
  continuous.  To prove that $U$ is bounded, we consider the action
  of~$U$ on a function of the form $f = \sum_{j=1}^n
  c_j\mathbf{1}_{[0,\alpha_j]}$, where $n\in\mathbb{N}$,
  $c_1,\ldots,c_n\in\mathbb{K}$ and
  $0\le\alpha_1<\alpha_2<\cdots<\alpha_n\le\omega_1$. We have $\|f\| =
  \max\bigl\{|\sum_{j=m}^n c_j| : 1\le m\le n\bigr\}$, while for
  $\beta\in[0,\omega_1]$,
  \[ (Uf)(\pi(\beta)) = \sum_{j\in J}
  c_j\mathbf{1}_{\pi(A_{\alpha_j})}(\pi(\beta)) = \sum_{j\in J}
  c_j\mathbf{1}_{A_{\alpha_j}}(\beta), \] where $J = \{
  j\in\{1,\ldots,n\} :\alpha_j\notin H\}$ and the second equality
  follows because $A_{\alpha_j}$ is disjoint from~$H$ for each $j\in
  J$. Thus $(Uf)(\pi(\beta)) = 0$ if $\beta\notin \bigcup_{j\in J}
  A_{\alpha_j}$. Now suppose that $\beta\in A_{\alpha_j}$ for some
  $j\in J$.  If $\beta\in[0,\omega^\omega)$, we let $\lambda = 0$, and
  otherwise we choose a limit ordi\-nal~$\lambda\in[\omega,\omega_1)$
  such that $\beta\in(\omega^\lambda,\omega^{\lambda+\omega})$. Then,
  letting \[ k = \min\{ j\in\{1,\ldots,n\} :\beta\le\alpha_j\}\qquad
  \text{and}\qquad m = \max\{ j\in\{1,\ldots,n\} :\alpha_j <
  \omega^{\lambda+\omega}\}, \] we have $\beta\in A_{\alpha_j}$ if and
  only if $k\le j\le m$, so
  \[ (Uf)(\pi(\beta)) = \sum_{j=k}^m c_j = \sum_{j=k}^n c_j -
  \sum_{j=m+1}^n c_j, \] and consequently $|(Uf)(\pi(\beta))|\le
  \bigl|\sum_{j=k}^n c_j\bigr| + \bigl|\sum_{j=m+1}^n c_j\bigr| \le
  2\|f\|$. This proves that $U$ is bounded with norm at most two. (In
  fact $\|U\| = 2$ because $f = -2\mathbf{1}_{\{0\}} +
  \mathbf{1}_{[0,\omega^\omega]}\in C([0,\omega_1])$ has norm one, so
  $\| U\|\ge \|Uf\| = 2\|\mathbf{1}_{\{\pi(0)\}}\| = 2$.)

  Since the subspace
  \mbox{$\lin\{\mathbf{1}_{[0,\alpha]}:\alpha\in[0,\omega_1]\}$} is
  dense in~$C([0,\omega_1])$, $U$ extends uniquely to an operator of
  norm two defined on~$C([0,\omega_1])$.  We now claim that the
  operator $V= C_\pi U$ belongs to~$\mathscr{M}\setminus%
  (\mathscr{X}(C([0,\omega_1])) +
  \overline{\mathscr{G}}_{c_0(\omega_1)}(C([0,\omega_1])))$. Once
  verified, this claim will complete the proof.

  We have $V\in\mathscr{M}$ because $k^V_{\omega_1} = 0$. Indeed,
  given $\alpha\in[0,\omega_1]$, we shall prove that
  $V_{\alpha,\omega_1} = 0$ by direct computation. Since $r^V_\alpha$
  has countable support, we can choose a non-zero countable limit
  ordinal $\lambda$ such that $V_{\alpha,\beta} = 0$ for each
  $\beta\in(\omega^\lambda,\omega_1)$. Then
  \[ V_{\alpha,\omega_1} = \sum_{\beta\in(\omega^\lambda,\omega_1]}
  V_{\alpha,\beta} = (V\mathbf{1}_{(\omega^\lambda,\omega_1]})(\alpha)
  = C_\pi(U\mathbf{1}_{[0,\omega_1]} -
  U\mathbf{1}_{[0,\omega^\lambda]})(\alpha) = 0, \] where the final
  equality follows from~\eqref{defnU} because $\omega_1$ and
  $\omega^\lambda$ both belong to~$H$.

  To show that $V\notin\mathscr{X}(C([0,\omega_1])) +
  \overline{\mathscr{G}}_{c_0(\omega_1)}(C([0,\omega_1]))$, assume the
  contrary, say $V = R+S$, where
  $R\in\mathscr{X}(C([0,\omega_1]))$ and
  $S\in\overline{\mathscr{G}}_{c_0(\omega_1)}(C([0,\omega_1]))$.  By
  Theorem~\ref{thmseprange}, we can choose a countable
  ordinal~$\sigma$ such that $R = \widetilde{P}_\sigma
  R\widetilde{P}_\sigma$, and thus
  \begin{equation}\label{eq1XplusGandM}
    (I - \widetilde{P}_\sigma)V = (I - \widetilde{P}_\sigma)S%
    \in\overline{\mathscr{G}}_{c_0(\omega_1)}(C([0,\omega_1])).
  \end{equation} Take a non-zero countable ordinal~$\tau$
  such that   $\sigma\le\omega^{\omega^\tau}$, and let   $\lambda =
  \omega^\tau$. Further, let $\iota\colon 
  (\omega^\lambda, \omega^\lambda\cdot 2]\to [0,\omega_1]$
  be the inclusion map, and define \mbox{$\rho\colon
    [0,\omega_1]\to (\omega^\lambda, \omega^\lambda\cdot 2]$} by
  \[ \rho(\alpha) = \begin{cases} \omega^\lambda +1 &\text{for}\ \
    \alpha\in [0,\omega^\lambda]\\ \alpha &\text{for}\ \
    \alpha\in(\omega^\lambda,\omega^\lambda\cdot 2)\\
    \omega^\lambda\cdot 2 &\text{for}\ \ \alpha\in[\omega^\lambda\cdot
    2,\omega_1]. \end{cases} \] Clearly $\rho$ is continuous, and we
  claim that the diagram
  \begin{equation}\label{eq2XplusGandM}
    \spreaddiagramrows{6ex}\spreaddiagramcolumns{1.75ex}%
    \xymatrix{ C((\omega^\lambda,
      \omega^\lambda\cdot 2])\ar[rrr]^{\displaystyle{I}}
      \ar[d]_{\displaystyle{C_\rho}} & & &
      C((\omega^\lambda, \omega^\lambda\cdot 2])\\
      C([0,\omega_1])\ar[r]^{\displaystyle{P_{\omega^\lambda\cdot 2}}} &
      C([0,\omega_1])\ar[r]^{\displaystyle{V}} &
      C([0,\omega_1])\ar[r]^{\displaystyle{I-\widetilde{P}_\sigma}} &
      C([0,\omega_1])\ar[u]_{\displaystyle{C_\iota}}} \end{equation}
  is commutative, where  $C_\rho\colon f\mapsto f\circ\rho$
  and $C_\iota\colon f\mapsto f\circ\iota$ denote the composition
  operators associated with~$\rho$  and~$\iota$, respectively. To
  verify this claim, it suffices to check the action on each function
  of the 
  form   $\mathbf{1}_{(\omega^\lambda,\alpha]}$, where $\alpha\in 
  (\omega^\lambda, \omega^\lambda\cdot 2]$, because such functions
  span a dense subspace of \mbox{$C((\omega^\lambda,
    \omega^\lambda\cdot 2])$}. We have $C_\rho
  \mathbf{1}_{(\omega^\lambda,\alpha]} = \mathbf{1}_{[0,\alpha]}$ for
  $\alpha\in(\omega^\lambda,\omega^\lambda\cdot 2)$ and $C_\rho
  \mathbf{1}_{(\omega^\lambda,\omega^\lambda\cdot 2]} =
  \mathbf{1}_{[0,\omega_1]}$, so $P_{\omega^\lambda\cdot 2}C_\rho
  \mathbf{1}_{(\omega^\lambda,\alpha]} = \mathbf{1}_{[0,\alpha]}$
  for each $\alpha\in(\omega^\lambda,\omega^\lambda\cdot
  2]$. Hence, by~\eqref{defnU},
  \[ C_\iota(I-\widetilde{P}_\sigma)VP_{\omega^\lambda\cdot 2}C_\rho
  \mathbf{1}_{(\omega^\lambda,\alpha]} =
  C_\iota(I-\widetilde{P}_\sigma)C_\pi \mathbf{1}_{\pi(A_\alpha)} =
  C_\iota(I-\widetilde{P}_\sigma) \mathbf{1}_{A_\alpha} =
  \mathbf{1}_{A_\alpha}, \] which proves the claim because $A_\alpha =
  (\omega^\lambda,\alpha]$.

  The map $\alpha\mapsto \omega^\lambda+1+\alpha,\,
  [0,\omega^\lambda]\to(\omega^\lambda, \omega^\lambda\cdot 2]$, is a
  homeo\-mor\-phism, so the Banach spaces $C([0,\omega^\lambda])$ and
  $C((\omega^\lambda,\omega^\lambda\cdot 2])$ are isometrically
  isomorphic.  Hence $C((\omega^\lambda, \omega^\lambda\cdot 2])$ has
  Szlenk index $\omega^{\tau+1}$ by
  Theorem~\ref{SzlenkCK}\romanref{SzlenkCK1}.

  On the other hand, Theorem~\ref{SzlenkCK}\romanref{SzlenkCK2}
  implies that $\overline{\mathscr{G}}_{c_0(\omega_1)}\subseteq
  \mathscr{S}\!\mathscr{Z}_1$, so
  by~\eqref{eq1XplusGandM}--\eqref{eq2XplusGandM} (the identity
  operator on) $C((\omega^\lambda, \omega^\lambda\cdot 2])$ has Szlenk
  index at most~$\omega$, contradicting the conclusion of the previous
  paragraph.
\end{proof}

\begin{remark}
  We shall here outline an alternative, more abstract, approach to
  part of the proof of Theorem~\ref{XplusGandM} given above as it
  sheds further light on a construction therein and raises an
  interesting question at the end.  Our starting point is the
  observation that the compact Hausdorff space~$K$ defined in the
  proof of Theorem~\ref{XplusGandM} is in fact just a convenient
  realization of the one-point compactification of the dis\-joint
  union of the intervals~$[0,\omega^{\omega^\alpha}]$
  for~$\alpha\in[0,\omega_1)$.

  A space is \emph{Eberlein compact} if it is homeomorphic to a weakly
  compact subset of~$c_0(\Gamma)$ for some index set~$\Gamma$.  Being
  compact metric spaces, the inter\-vals~$[0,\omega^{\omega^\alpha}]$
  are Eberlein compact whenever~$\alpha$ is countable. Therefore, by a
  result of Lindenstrauss~\cite[Proposition~3.1]{lin}, the one-point
  compactification of their disjoint union is Eberlein compact; that
  is, our space~$K$ is Eberlein compact. On the other hand, the
  interval~$[0,\omega_1]$ is not Eberlein compact.

  A Banach space~$X$ is \emph{weakly compactly generated} if it
  contains a weakly compact subset~$W$ such that $X = \cllin W$. Amir
  and Lindenstrauss~\cite{al} have shown that a compact space~$L$ is
  Eberlein compact if and only if the Banach space~$C(L)$ is weakly
  compactly generated. Hence, returning to our case, we see that
  $C(K)$ is weakly compactly generated, whereas $C([0,\omega_1])$ is
  not. This implies that the closed
  ideal~$\overline{\mathscr{G}}_{C(K)}(C([0,\omega_1]))$ is proper and
  thus contained in the Loy--Willis ideal~$\mathscr{M}$. By
  definition, the operator~$V$ defined in the proof of
  Theorem~\ref{XplusGandM} factors through $C(K)$. On the other hand,
  we showed there that it does not belong to
  $\mathscr{X}(C([0,\omega_1])) +
  \overline{\mathscr{G}}_{c_0(\omega_1)}(C([0,\omega_1]))$, so this
  ideal is distinct
  from~$\overline{\mathscr{G}}_{C(K)}(C([0,\omega_1]))$.

  To prove that $\mathscr{X}(C([0,\omega_1])) +
  \overline{\mathscr{G}}_{c_0(\omega_1)}(C([0,\omega_1]))$ is
  contained in $\overline{\mathscr{G}}_{C(K)}(C([0,\omega_1]))$,
  consider first an ordinal $\lambda$ of the form~$\omega^\tau$,
  where~$\tau\in[1,\omega_1)$.  Replacing $\widetilde{P}_\sigma$
  with~$0$ in~\eqref{eq2XplusGandM}, we obtain a commutative diagram
  as before, and since $V$ factors through $C(K)$ and
  $C([0,\omega^\lambda])\cong C((\omega^\lambda,\omega^\lambda\cdot
  2])$, we conclude that the identity operator on
  $C([0,\omega^\lambda])$ factors through~$C(K)$.  Hence $C(K)$
  contains a complemented copy of $C([0,\omega^\lambda])$, and
  therefore we have \mbox{$\mathscr{X}(C([0,\omega_1]))\subseteq
    \mathscr{G}_{C(K)}(C([0,\omega_1]))$} by
  Theorem~\ref{thmseprange}.  Secondly, Lemma~\ref{lemma28june1}
  implies that $(\mathbf{1}_{\{\pi(\omega^\lambda+1)\}})$, where
  $\lambda$ ranges over all non-zero countable limit ordinals, is a
  transfinite basic sequence in~$C(K)$ equivalent to the canonical
  Schauder basis for~$c_0(\omega_1)$.
  Proposition~\ref{argyrosetalProp3.2} ensures that the closed linear
  span of this sequence is complemented in $C([0,\omega_1])$ and hence
  also in the subspace~$C(K)$, so
  $\mathscr{G}_{c_0(\omega_1)}\subseteq \mathscr{G}_{C(K)}$.  

  Thus, to summarize, we have shown that
  \[ \mathscr{X}(C([0,\omega_1])) +
  \overline{\mathscr{G}}_{c_0(\omega_1)}(C([0,\omega_1]))\subsetneq
  \overline{\mathscr{G}}_{C(K)}(C([0,\omega_1]))\subseteq\mathscr{M}. \]
  We do not know whether the final inclusion is proper; we conjecture
  that it is.

  Another interesting question is whether the inclusion
  \[ \overline{\mathscr{G}}_{C(K_{\alpha})\oplus
    c_0(\omega_1)}(C([0,\omega_1]))\subseteq
  \mathscr{S}\!\mathscr{Z}_{\alpha+1}(C([0,\omega_1])), \] established
  in the proof of Proposition~\ref{propinfinitechains}, is proper for
  some, or each, countable ordinal~$\alpha$.
\end{remark}

\bibliographystyle{amsplain}

\end{document}